\documentclass[final,leqno,hidelinks]{siamltex}

\usepackage{graphicx,amsmath,amsfonts,latexsym,amssymb,hyperref,comment,subeqnarray}
\usepackage{hyperref,mathrsfs}
\usepackage[mathscr]{euscript}
\usepackage{colortbl}
\newtheorem{example}{Example}[section]
\newtheorem{remark}{Remark}[section]
\numberwithin{equation}{section}

\DeclareFontEncoding{FMS}{}{}
\DeclareFontSubstitution{FMS}{futm}{m}{n}
\DeclareFontEncoding{FMX}{}{}
\DeclareFontSubstitution{FMX}{futm}{m}{n}
\DeclareSymbolFont{fouriersymbols}{FMS}{futm}{m}{n}
\DeclareSymbolFont{fourierlargesymbols}{FMX}{futm}{m}{n}
\DeclareMathDelimiter{\tbar}{\mathord}{fouriersymbols}{152}{fourierlargesymbols}{147}

\newcommand{\jump}[1]{\left[\!\left[#1\right]\!\right]}
\newcommand{\bigjump}[1]{\left[\!\!\left[#1\right]\!\!\right]}

\def\NN{\hbox{\rlap{I}\kern.16em N}}
\def\NC{\hbox{\rlap{\kern.24em\raise.1ex\hbox
                  {\vrule height1.3ex width.9pt}}C}}

\def \bZ {\mathbb Z}

\title{Superconvergence of Immersed Finite Volume Methods for one-dimensional Interface Problems}

\author{
Waixiang Cao\thanks{School of Data and Computer Science, Sun Yat-Sen University, Guangzhou, Guangdong 51006, China (caowx5@mail.sysu.edu.cn).}
\and
Xu Zhang\thanks{Department of Mathematics and Statistics, Mississippi State University, Mississippi State, MS 39762, USA (xuzhang@math.msstate.edu).}
\and
Zhimin Zhang\thanks{Beijing Computational Science Research Center, Beijing 100193, China (zmzhang@csrc.ac.cn); and Department of Mathematics, Wayne State University, Detroit, MI 48202, USA (zzhang@math.wayne.edu).}
\and
Qingsong Zou\thanks{School of Data and Computer Science, Sun Yat-Sen University, Guangzhou, Guangdong 51006, China
(mcszqs@mail.sysu.edu.cn).}
}

\begin{document}
\thanks{The work of W. Cao was supported in part by
the NSFC grant 11501026, and the China Postdoctoral Science Foundation grant 2016T90027, 2015M570026.
The work of Z. Zhang was supported in part by
the NSFC grants 11471031, 91430216, and U1530401; and NSF grant DMS-1419040.
The work of Q. Zou was supported in part by
the NSFC grants  11571384 and 11428103,  by Guangdong NSF grant 2014A030313179 and by Fundamental Research Funds for the Central Universities  grant
16lgjc80.}
\maketitle

\begin{abstract}
In this paper, we introduce a class of high order immersed finite volume methods (IFVM) for one-dimensional interface problems. 
We show the optimal convergence of IFVM in $H^1$- and $L^2$- norms. We also prove some superconvergence results of IFVM. 
To be more precise, the IFVM solution is superconvergent of order $p+2$ at the roots of generalized Lobatto polynomials, and the flux is superconvergent of order $p+1$ at generalized Gauss points on each element including the interface element. Furthermore, for diffusion interface problems, the convergence rates for IFVM solution at the mesh points and the flux at generalized Gaussian points can both be raised to $2p$. These superconvergence results are consistent with those for the standard finite volume methods. Numerical examples are provided to confirm our theoretical analysis. 
\end{abstract}



\pagestyle{myheadings}
\thispagestyle{plain}

\section{Introduction}
Interface problems arise in many simulations in science and engineering that involve multi-physics and multi-materials. Classical numerical methods, such as finite element methods (FEM) \cite{1994BrennerScott,1998ChenZou,1982Xu}, and finite volume methods (FVM) \cite{Bank.R;Rose.D1987,Barth.T;Ohlberger2004,
Cai.Z1991, Cai.Z_Park.M2003, ChenWuXu2011, Ewing.R;Lin.T;Lin.Y2002,
EymardGallouetHerbin2000, Li.R2000,
Ollivier-Gooch;M.Altena2002,Plexousakis_2004,Suli1991,
Xu.J.Zou.Q2009,ZhangZou_NM_2015} usually require solution meshes to fit the interface; otherwise, the convergence may be impaired. The immersed finite element methods (IFEM) \cite{2007AdjeridLin, 2009AdjeridLin, 1998Li} are a class of FEM that relax the body-fitting requirement, hence Cartesian meshes can be used for solving interface problems with arbitrary interface geometry. The key ingredient of IFEM is to design some special basis functions on interface elements that can capture the non-smoothness of the exact solution. Recently, this \textit{immersed} idea has also been used in a variety of numerical schemes such as conforming FEM \cite{2008HeLinLin,2003LiLinWu,2015LinLinZhang}, nonconforming FEM \cite{2010KwakWeeChang,2013LinSheenZhang,2015LinSheenZhang}, discontinuous Galerkin methods \cite{2010HeLinLin,2015LinYangZhang1,2016YangZhang}, and FVM \cite{1999EwingLiLinLin,2009HeLinLin}. 

The use of structured mesh, especially Cartesian meshes, often leads to some superconvergence phenomenon. The superconvergence is a phenomenon that the order of convergence at certain points surpass the maximum order of convergence of the numerical schemes. There has been a growing interest in the study of superconvergence, for example, finite element methods \cite{Babuska1996,Bramble.Schatz.math.com,Chen.C.M2012,Neittaanmaki1987,V.Thomee.math.comp,1995Wahlbin},
 finite volume methods \cite{Cai.Z1991, Cao;Zhang;Zou2012,Cao;zhang;zou:2kFVM, Chou_Ye2007,Xu.J.Zou.Q2009}, discontinuous Galerkin and local discontinuous Galerkin methods \cite{Adjerid;Massey2006,Cao-Shu-zhang-Yang2D,
Cao;zhang:supLDG2k+1,Cao;zhang;zou:2k+1,Guo_zhong_Qiu2013,Xie;Zhang2012,Yang;Shu:SIAM2012}.

In this article, we first introduce a class of high order IFVM for one dimensional interface problems. Thanks to the unified construction of FVM schemes in \cite{Cao;Zhang;Zou2012, ZhangZou_NM_2015} and the generalized orthogonal polynomials developed in \cite{2017CaoZhangZhang}, we can develop the high order IFVM in a systematical approach. To be more specific, we adopt the standard $p$-th degree IFE spaces \cite{2007AdjeridLin, 2009AdjeridLin, 2017CaoZhangZhang} as our trial function space. Using the roots of generalized Legendre polynomials, known as generalized Gauss points, as the control volume, we construct the test function space as the piecewise constant corresponding to the dual meshes. The advantage of our IFVM is that it does not require the mesh to be aligned with the interface, and it inherits all the desired properties of the classical FVM such as local conservation of flux.

The main focus of this article is the error analysis of IFVM, especially the superconvergence analysis. By establishing the inf-sup condition and continuity of the bilinear form, we prove that our IFVM converge optimally in $H^1$- norm. As for the superconvergence,
we prove that the immersed finite volume (IFV) solution is superconvergent of the order $O(h^{p+2})$ at the generalized Lobatto points on both non-interface and interface elements, and the flux error is superconvergent at the generalized Gauss points of the order $O(h^{p+1})$. The error of IFV solution and the Gauss-Lobatto projection is superclose. In particular, for the diffusion interface problem, we show that the convergence rate of both the solution error at nodes and the flux error at Gauss points can be enhanced to $O(h^{2p})$. All these results are consistent with the superconvergence analysis of the standard FVM in \cite{Cao;Zhang;Zou2012}.

However, there is a significant difference in the superconvergence analysis of IFVM compared with the analysis of standard FVM \cite{Cao;Zhang;Zou2012}. Due to the low global regularity of the exact solution, the standard approach using the Green function cannot be directly applied to the IFVM for interface problems. The key ingredient in the analysis is the construction of generalized Lobatto points and a specially designed interpolation function. That is, we first choose a class of generalized Lobatto polynomials as our basis functions  that satisfy both orthogonality and interface jump conditions, then we use these orthogonal basis function to design a special interpolant of the exact solution which is superclose to the IFV solution. The supercloseness of the interpolation and the IFV solution yields the desired superconvergence results for the IFV solution.

The rest of the paper is organized as follows.
In Section 2 we recall the generalized orthogonal polynomials and present the high order IFVM for interface problems in one-dimensional setting. 
In Section 3 we provide a unified analysis for the inf-sup condition and establish the optimal convergence in $H^1$ norm.
In Section 4, we study the superconvergence property of IFVM. We identify and analyze superconvergence points for the IFV solution at both interface and non-interface elements. Numerical examples are presented in Section 5. Finally, some concluding remarks are summarized in Section 6. 

In the rest of this paper, we use the notation``$A\lesssim B$" to denote $A$ can be bounded by $B$ multiplied by a constant independent of the mesh size. Moreover, ``$A\sim B$" means $``A\lesssim B"$ and $``B\lesssim A"$.

\section{Interface Problems and Immersed Finite Volume Methods}
Assume that $\Omega = (a,b)$ is an open interval in $\mathbb{R}$. Let $\alpha\in \Omega$ be an interface point such that $\Omega^- = (a,\alpha)$ and $\Omega^+ = (\alpha,b)$. Consider the following one-dimensional elliptic interface problem
\begin{equation}\label{eq: DE}
  -(\beta u')' + \gamma u' + cu = f, ~~~x\in \Omega^-\cup\Omega^+,
\end{equation}
\begin{equation}\label{eq: BC}
  u(a) = u(b) = 0.
\end{equation}
Here, the coefficients $\gamma$ and $c$ are assumed to be constants. The diffusion coefficient $\beta$ has a finite jump across the interface. Without loss of generality, we assume it is a piecewise constant function
\begin{equation}\label{eq: jump}
  \beta(x) =
  \left\{
    \begin{array}{ll}
      \beta^-, & \text{if}~ x\in\Omega^-, \\
      \beta^+, & \text{if}~ x\in\Omega^+,
    \end{array}
  \right.
\end{equation}
where $\beta_0=\min\{\beta^+,\beta^-\}>0$. At the interface $\alpha$, the solution is assumed to satisfy the interface jump conditions
\begin{equation}\label{eq: jump condition}
  \jump{u(\alpha)} = 0,~~~
  \bigjump{\beta u'(\alpha)} = 0,
\end{equation}
where $\jump{v(\alpha)} = \lim\limits_{x\to\alpha^+}v(x) - \lim\limits_{x\to\alpha^-}v(x)$.

\subsection{Generalized orthogonal polynomials}
First, we briefly review the generalized Legendre and Lobatto polynomials developed in \cite{2017CaoZhangZhang}. These generalized orthogonal polynomials will be used to form the trial function space in the IFVM.

Let $\tau=[-1,1]$ be the reference interval, and $P_n(\xi)$ be the standard Legendre polynomial of degree $n$ on $\tau$ satisfying the following orthogonality condition
 \begin{equation}\label{eq: Legendre orthogonality}
  \int_{-1}^1 P_m(\xi)P_n(\xi) d\xi = \frac{2}{2n+1}\delta_{mn}.
\end{equation}
Define a family of Lobatto polynomials $\{\psi_n\}$ on $\tau$ as follows
\begin{equation}\label{eq: Lobatto Poly}
  \psi_0(\xi) = \frac{1-\xi}{2},~~~
  \psi_1(\xi) = \frac{1+\xi}{2},~~~
  \psi_n(\xi) = \int_{-1}^\xi P_{n-1}(t)dt, ~~~ n\geq 2.
\end{equation}
The generalized Legendre polynomials $\{L_n\}$ on $\tau$ with a discontinuous weight is defined as
\begin{equation}\label{eq: general Legendre orthogonality}
  (L_n,L_m)_w :=\int_{-1}^1 w(\xi) L_n(\xi) L_m(\xi) d\xi = c_n\delta_{mn},
\end{equation}
where $w(\xi) =  \frac{1}{\hat\beta(\xi)}$ and
\begin{equation}\label{eq: discontinuous weight}
  \hat\beta(\xi) =
  \left\{
    \begin{array}{ll}
      \beta^-, & \text{if}~ \xi\in\tau^-=(-1,\hat\alpha), \\
      \beta^+, & \text{if}~ \xi\in\tau^+=(\hat\alpha,1).
    \end{array}
  \right.
\end{equation}
The generalized Lobatto polynomials $\{\phi_n\}$ can be constructed in a similar manner as \eqref{eq: Lobatto Poly} as follows:
\begin{eqnarray}
  \phi_0(\xi) &=& \left\{\begin{array}{ll}
      \frac{(1-\hat\alpha)\beta^-+(\hat\alpha-\xi)\beta^+}{(1-\hat\alpha)\beta^-+(1+\hat\alpha)\beta^+}, & \text{in}~\tau^-, \vspace{1mm}\\
      \frac{(1-\xi)\beta^-}{(1-\hat\alpha)\beta^-+(1+\hat\alpha)\beta^+}, & \text{in}~ \tau^+.
    \end{array}\right. \label{eq: lobatto 0}\\
  \phi_1(\xi) &=& \left\{
  \begin{array}{ll}
      \frac{(1+\xi)\beta^+}{(1-\hat\alpha)\beta^-+(1+\hat\alpha)\beta^+}, & \text{in}~\tau^-, \vspace{1mm}\\
      \frac{(\xi-\hat\alpha)\beta^-+(1+\hat\alpha)\beta^+}{(1-\hat\alpha)\beta^-+(1+\hat\alpha)\beta^+}, & \text{in}~\tau^+.
    \end{array}\right. \label{eq: lobatto 1}\\
  \phi_{n}(\xi) &=&  \int_{-1}^\xi w(t) L_{n-1}(t)dt,~~~n \geq 2.\label{eq: lobatto n}
\end{eqnarray}

These generalized orthogonal polynomials can be used as local basis functions on interface element, as they satisfy both the orthogonality and interface jump conditions:
\[
    \bigjump{ \phi_n(\hat\alpha)} = 0,\ \ \ \bigjump{\hat \beta \phi_n^{(j)}(\hat\alpha)} = 0,~~~~\forall~ j = 1,2, \cdots, n.
\]
Note that the generalized Legendre polynomials are polynomials, but the generalized Lobatto polynomials are piecewise polynomials. As pointed out in \cite{2017CaoZhangZhang}, the generalized orthogonal polynomials can be explicitly constructed. In Figure \ref{fig: lobatto IFE basis}, we plot the first few generalized orthogonal polynomials for $\hat\beta = [1,5]$, and the reference interface point $\hat\alpha = 0.15$. For comparison, we also plot the standard Legendre and Lobatto polynomials in Figure \ref{fig: lobatto FE basis}. We note that these functions are consistent with the generalized orthogonal polynomials when $\beta^+ = \beta^-$, as stated in Lemma 3.2 in \cite{2017CaoZhangZhang}.
\begin{figure}[thb]
  \centering
  \includegraphics[width=.4\textwidth]{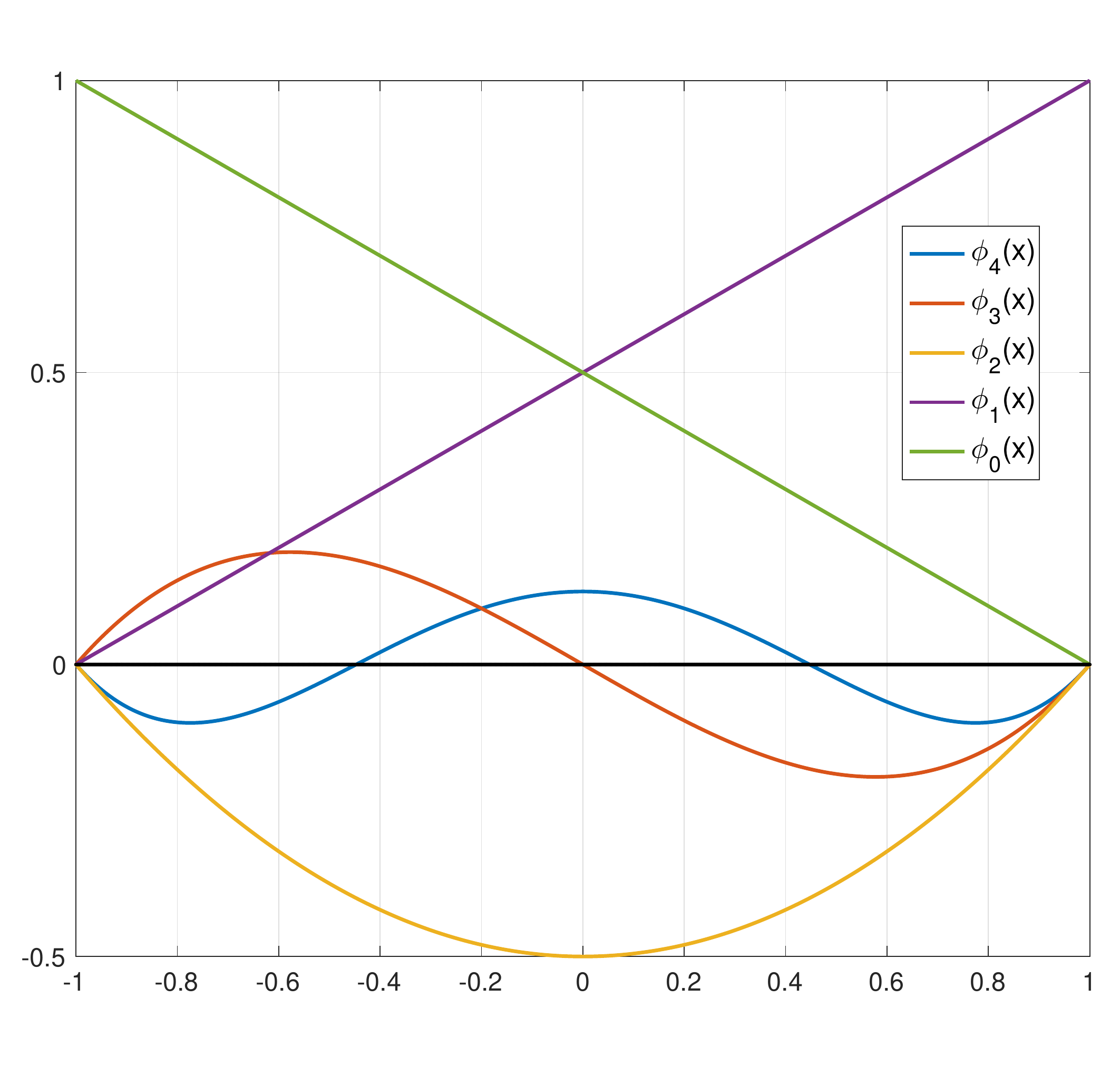}
  \includegraphics[width=.4\textwidth]{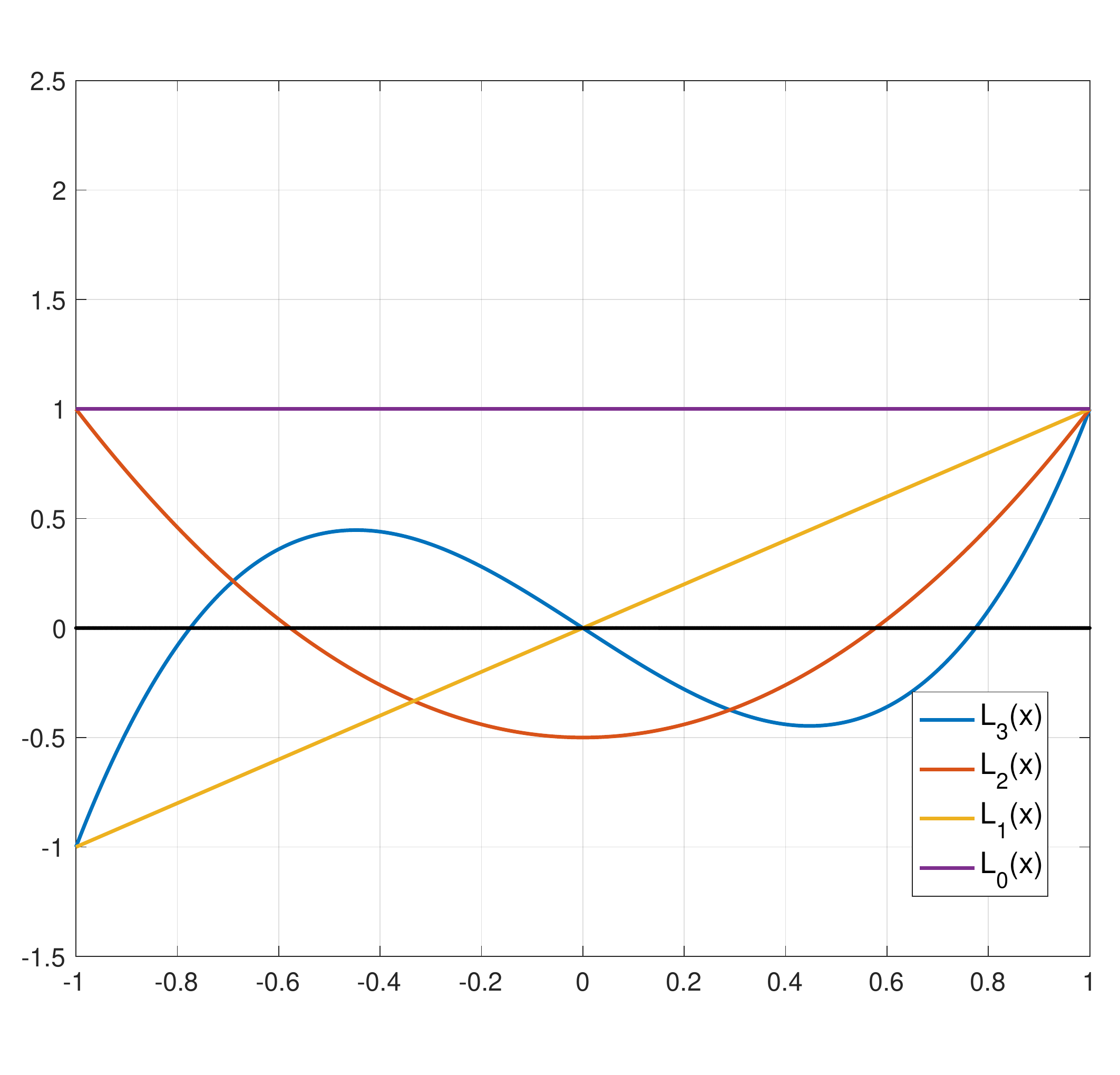}\\
  \caption{Standard Lobatto (left) and Legendre (right) polynomials}
  \label{fig: lobatto FE basis}
\end{figure}

\begin{figure}[thb]
  \centering
  \includegraphics[width=.4\textwidth]{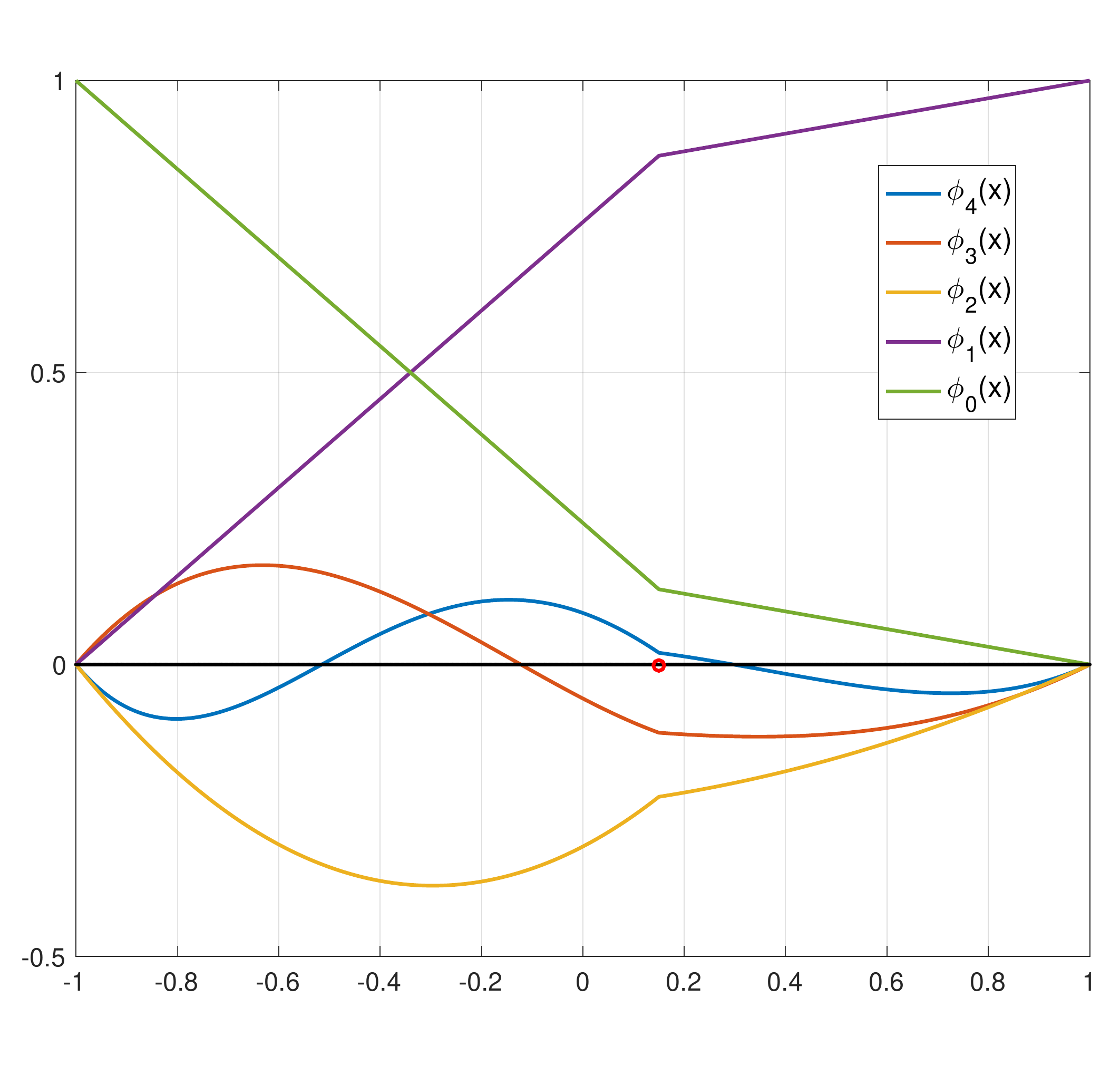}
  \includegraphics[width=.4\textwidth]{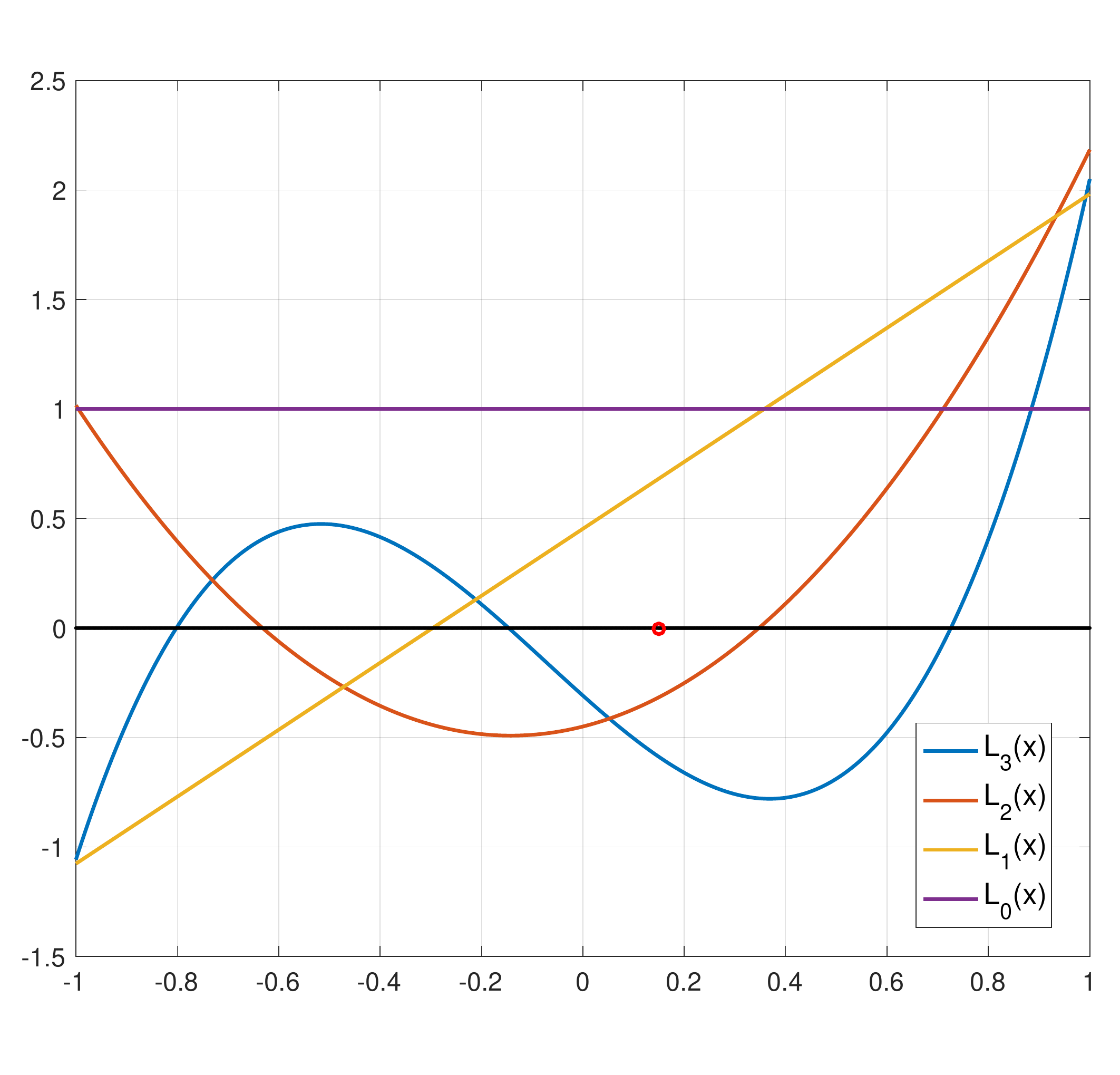}\\
  \caption{Generalized Lobatto (left) and Legendre (right) polynomials with interface $\hat\alpha = 0.15$}
  \label{fig: lobatto IFE basis}
\end{figure}

\subsection{Immersed finite volume methods}
In the subsection, we introduce the immersed finite volume methods
for solving the interface problem \eqref{eq: DE} - \eqref{eq: jump condition}.
Consider the following partition of $\Omega$, independent of interface
\begin{equation}\label{eq: partition}
  a = x_0 < x_1< \cdots < x_{k-1} <\alpha< x_{k}<\cdots<x_N = b.
\end{equation}
For a positive integer N, let $\bZ_N:=\{1,\cdots,N\}$ and for all $i \in \bZ_N$, we denote $\tau_i=[x_{i-1},x_i]$ and  $h_i=x_i-x_{i-1}$, $h=\max\limits_{i\in \bZ_N}h_i$.  Let
$\mathcal{T}= \{\tau_i\}_{i = 1}^N$ be a partition of $\Omega$, and we suppose the partition is shape regular, i.e., the ratio between the maximum 
and minimum mesh sizes shall stay bounded during mesh refinements.
   We call the element $\tau_k$ the interface element since it
   contains the interface point $\alpha$,  and the rest of elements $\tau_i$, $i\ne k$ noninterface elements.

The basis functions of the trial function space is constructed using the (generalized) Lobatto polynomials. In fact, we define the basis functions in each element $\tau_i, i\in\bZ_N$ as
\begin{eqnarray}
 \phi_{i,n}(x) &=& \left\{
  \begin{array}{ll}
      \psi_n(\xi) = \psi_n\left(\frac{2x-x_{i-1}-x_i}{h_{i}}\right), & i\neq k, \vspace{1mm}\\
      \phi_n(\xi) = \phi_n\left(\frac{2x-x_{k-1}-x_k}{h_{k}}\right), & i=k.
    \end{array}\right.
 \end{eqnarray}
Then  corresponding trial function space is defined by
\begin{equation}\label{eq: IFE space}
 U_{\cal T}:= \{v\in C(\Omega): v|_{\tau_i} \in \text{span}\{\phi_{i,n}: n = 0,1,\cdots, p\}, v(a)=v(b)=0\}.
\end{equation}
Obviously, $\dim U_{\cal T}=Np-1$.

Next we present the dual partition and its corresponding test function space. It has been shown in \cite{2017CaoZhangZhang} that the generalized Legendre polynomials $\{L_n\}$ and generalized Lobatto polynomials $\{\phi_n\}$ have same numbers of roots as the standard Legendre polynomials $\{P_n\}$ and Lobatto polynomials $\{\psi_n\}$. Let
  \begin{eqnarray}
 P_{i,n}(x) &=& \left\{
  \begin{array}{ll}
      P_n(\xi) = P_n\left(\frac{2x-x_{i-1}-x_i}{h_{i}}\right), & i\neq k, \vspace{1mm}\\
      L_n(\xi) = L_n\left(\frac{2x-x_{k-1}-x_k}{h_{k}}\right), & i=k.
    \end{array}\right.
 \end{eqnarray}
We denote by $g_{i,j}, j\in\bZ_n$ the (generalized) Gauss points of degree $n$ in $\tau_i$. That is, the $n$ roots of $P_{i,n}$. With these Gauss points,  we construct a dual partition
\[
  {\cal T}'=\{\tau'_{1,0}, \tau'_{N,p}\}\cup \{\tau'_{i,j}: (i,j)\in \bZ_N \times \bZ_{p_i}\},
\]
  where
\[
  \tau'_{1,0}=[a,g_{1,1}], \tau'_{N,p}=[g_{N,p},b],\tau'_{i,j}=[g_{i,j}, g_{i,j+1}],
\]
here
\[
  p_i=\left\{\begin{array}{lll}
 p& \text{if} &i \in \bZ_{N-1}\\
p-1 & \text{if} &i=N
\end{array}
\right.
\text{and} \quad g_{i,p+1}=g_{i+1,1},\forall i \in\bZ_{N-1}.
\]
The test function space $V_{\cal T'}$ consists of the piecewise constant functions with respect
to the partition $\cal T'$, which vanish on the intervals $\tau'_{1,0}\cup \tau'_{N,p}$.
In other words,
\[
  V_{\cal T'}=\text{Span}\left\{\varphi_{i,j}: (i,j)\in \bZ_N \times \bZ_{p_i}\right\},
\]
where $\varphi_{i,j}=\chi_{[g_{i,j}, g_{i,j+1}]} $ is the characteristic function on the interval $\tau'_{i,j}$. We find that $\dim V_{\cal T'}=Np-1=\dim U_{\cal T}$.
The IFVM for solving \eqref{eq: DE} - \eqref{eq: jump condition} is: find $u_{\cal T}\in U_{\cal T}$ such that
\begin{eqnarray}
\beta(g_{i,j})u_{\cal T}'(g_{i,j})-\beta(g_{i,j+1})u_{\cal T}'(g_{i,j+1}) &+&\int_{g_{i,j}}^{g_{i,j+1}}\big( \gamma u_{\cal T}'(x)+cu_{\cal T}(x)\big)dx\nonumber \\
 &=& \int_{g_{i,j}}^{g_{i,j+1}}f(x){dx},~~~\forall (i,j)\in \bZ_N\times \bZ_{p_i}.\label{conserve}
\end{eqnarray}

   Given a function $v_{\cal T'}\in V_{\cal T'}$,  $v_{\cal T'}$ can be represented as
\[
   v_{\cal T'}=\sum_{i=1}^{N}\sum_{j=1}^{p_i}v_{i,j}\varphi_{i,j},
\]
 where  $v_{i,j}, (i,j)\in\bZ_N\times\bZ_{p_i}$ are constants. Multiplying \eqref{conserve} with $v_{i,j}$ and then summing up all $i,j$, we obtain
\begin{equation*}
\begin{split}
 \sum_{i=1}^{N}\sum_{j=1}^{p_i} v_{i,j}\left((\beta u'_{\cal T})(g_{i,j})-(\beta u_{\cal T}')(g_{i,j+1})+\int_{g_{i,j}}^{g_{i,j+1}}\big(\gamma u_{\cal T}'(x)+cu_{\cal T}(x)\big)dx \right)
 =\int_a^b f(x)v_{\cal T'}(x) dx,
 \end{split}
\end{equation*}
  or equivalently,
\begin{equation*}
   \sum_{i=1}^{N}\sum_{j=1}^{p} [v_{i,j}](\beta u_{\cal T}')(g_{i,j})+\sum_{i=1}^{N}\sum_{j=1}^{p_i} v_{i,j}\left(\int_{g_{i,j}}^{g_{i,j+1}}\big(\gamma u_{\cal T}'(x)+cu_{\cal T}(x)\big)dx \right)
   =\int_a^b f(x)v_{\cal T'}(x) dx,
\end{equation*}
   where $[v_{i,j}]=v_{i,j}-v_{i,j-1}$ is the jump of $v$ at the point $g_{i,j}, (i,j)\in \bZ_N \times \bZ_p$ with $v_{1,0}=0, v_{N,p}=0$ and $v_{i,0}=v_{i-1,p},2\leq i\leq N$.

The bilinear form of IFVM can be written as
\begin{equation}\label{bilinear1}
   a(u,v_{\cal T'})=\sum_{i=1}^{N}\sum_{j=1}^{p} [v_{i,j}]\beta(g_{i,j})u'(g_{i,j})
   +\sum_{i=1}^{N}\sum_{j=1}^{p_i} v_{i,j}\left(\int_{g_{i,j}}^{g_{i,j+1}}\big(\gamma u'(x)+cu(x)\big)dx \right),
\end{equation}
for all $u\in H_0^1(\Omega), v_{\cal T'}\in V_{\cal T'}$.
  Then our IFVM for the interface problem \eqref{eq: DE} - \eqref{eq: jump condition} can be rewritten as: Find $u_{\cal T}\in U_{\cal T}$ such that
\begin{equation}\label{eq: IFE method}
  a(u_{\cal T},v_{\cal T'})= (f,v_{\cal T'}),~~~~\forall v_{\cal T'}\in  V_{\cal T'}.
\end{equation}

\section{Convergence analysis}
    In this section, we derive the error estimation for IFVM. Following the same idea as in \cite{Cao;Zhang;Zou2012}, we first prove the inf-sup condition and continuity of the IFVM, and then use them to establish the optimal convergence rate of the IFV approximation.

\subsection{Inf-sup condition}

  We begin with some preliminaries. First,
   for any sub-domain $\Lambda\subset\Omega$, where $\Lambda^{\pm}=\Lambda\cap\Omega^{\pm}$,
   we define the following Sobolev spaces for $m\ge 1$ and $q\ge 1$ in $\Lambda$ as
\begin{eqnarray}
  \tilde W_\beta^{m,q}(\Lambda)= \Big\{v\in C(\Lambda)\ \!\!\!\!&:&\!\!\! \ \ \ v|_{\Lambda^\pm}\in W^{m,q}(\Lambda^\pm),~v|_{\partial\Omega\cap \Lambda}=0, \nonumber\\
   && \!\!\!\bigjump{\beta v^{(j)}(\alpha)} = 0,~ j = 1,2,\cdots, m\Big\} \label{eq: space2}
\end{eqnarray}
   equipped the norm and semi-norm
\[
    \|v\|^q_{m,q,\Lambda}=\|v\|^q_{m,q,\Lambda^-}+\|v\|^q_{m,q,\Lambda^+},\ \
    |v|^q_{m,q,\Lambda}=|v|^q_{m,q,\Lambda^-}+|v|^q_{m,q,\Lambda^+}.
\]
   If $\Lambda=\Omega$, we usually write $\|\cdot\|_{m,q}$ instead of  $\|\cdot\|_{m,q,\Omega}$, and
    $|\cdot|_{m}, \|\cdot\|_m$  instead of $|\cdot|_{m, 2}, \|\cdot\|_{m,2}$ when $q=2$ for simplicity.
   Second, we define a discrete energy norm for all $v\in H^1(\Omega)$ by
\[
    \|v\|^2_{G}=|v|_{G}^2+\|v\|_1^2,\ \ \
    |v|_{G}^2=\sum_{i=1}^N\sum_{j=1}^p A_{i,j}(\beta v'(g_{i,j}))^2.
\]
    Here $A_{i,j}, (i,j)\in\bZ_N\times\bZ_p$ are the weights of the Gauss quadrature
\[
Q_p(F)=\sum_{j=1}^p A_{i,j} F(g_{i,j})
\]
for computing the integral
\[
I(F)=\int_{\tau_i} w(x) F(x) dx=\int_{\tau_i} \frac{1}{\beta(x)} F(x) dx.
\]
For all $v_{\cal T'}\in V_{h},\ v_{\cal T'}=\sum\limits_{i=1}^N\sum\limits_{j=1}^{p_i} v_{i,j}\varphi_{i,j}$,
we let
\[
\big|v_{\cal T'}\big|^2_{1,\cal T'}=\sum_{i=1}^{N}\sum_{j=1}^p h_i^{-1}[v_{i,j}]^2,\quad
 \big\|v_{\cal T'}\big\|^2_{0,\cal T'}= \sum_{i=1}^{N}\sum_{j=1}^{p_i} h_i v_{i,j}^2,
\]
and
\[
     \big\|v_{\cal T'}\big\|_{\cal T'}^2 = \big|v_{h}\big|_{1,\cal T'}^2 + \big\|v_{\cal T'}\big\|_{0,\cal T'}^2.
\]
Also, we define a linear mapping $\Pi_{h}:U_{\cal T} \rightarrow V_{\cal T'}$ by
\[
    v_{\cal T'}=\Pi_{h}v_{\cal T}=\sum_{i=1}^N \sum_{j=1}^{p_i} v_{i,j}\varphi_{i,j},
\]
where the coefficients $v_{i,j}$ are determined by the constraints
\begin{equation}\label{F:1}
    [v_{i,j}]=A_{i,j}(\beta v_{\cal T}')(g_{i,j}),\ \ (i,j)\in \bZ_N\times \bZ_{p_i}.
\end{equation}

\begin{lemma}  For any $v_{\cal T}\in U_{\cal T}$, there holds
\begin{equation}\label{equi:norm}
   \|v_{\cal T}\|_{1}\thicksim\|v_{\cal T}\|_{G},\ \   \|\Pi_hv_{\cal T}\|_{\cal T'}\lesssim \|v_{\cal T}\|_{1}.
\end{equation}
\end{lemma}
 \begin{proof}   Noticing that $ (\beta v'_{\cal T})^2\in \mathbb P_{2p-2}$ for all $v_{\cal T}\in U_{\cal T}$, and
the $p$-point Gauss quadrature is exact for all polynomials of degree up to $2p-1$,  we obtain
\begin{equation}\label{eq:2}
   \sum_{i=1}^N\int_{\tau_i}\beta(x)( v'_{\cal T})^2(x)dx=\sum_{i=1}^N\int_{\tau_i}w(x)(\beta v'_{\cal T})^2(x)dx=\sum_{i=1}^N\sum_{j=1}^pA_{i,j}(\beta v'_{\cal T})^2(g_{i,j}).
\end{equation}
   Then the first inequality \eqref{equi:norm} follows.

   Denote $v_{1,0}=0$.  It follows from a direct calculation that
\[
   v_{i,j}=\sum_{m=1}^{i}\sum_{n=0}^j[v_{m,n}],
\]
    and thus
\[
   v^2_{i,j}\le p(b-a)\sum_{m=1}^{N}\sum_{n=0}^ph_m^{-1}[v_{m,n}]^2.
\]
   Then
\[
   \|\Pi_hv_{\cal T}\|_{0,\cal T'}\le p(b-a)|\Pi_hv_{\cal T}|_{1,\cal T'}
\]
  On the other hand,   for all $v_{\cal T}\in U_{\cal T}$, the derivative $\beta v'_{\cal T}\in \mathbb P_{p-1}(\tau_i), i\in\bZ_N$, then
\[
\sum_{i=1}^N\sum_{j=1}^{p} A_{i,j}\beta v_{\cal T}'(g_{i,j})=\int_a^b (w\beta v_{\cal T}')(x) dx=( v_{\cal T})(b)-( v_{\cal T})(a)=0.
\]
Therefore,
\begin{eqnarray*}
v_{N,p-1}&=&\sum_{i=1}^N\sum_{j=1}^{p_i}[v_{i,j}]
=\sum_{i=1}^N\sum_{j=1}^{p} A_{i,j}\beta v'_{\cal T}(g_{i,j})-A_{N,p}\beta v'_{\cal T}(g_{N,p})=-A_{N,p}v'_{\cal T}(g_{N,p}).
\end{eqnarray*}
In other words, we also have
\begin{equation}\label{F:2}
[v_{N,p}]=v_{N,p}-v_{N,p-1}=A_{N,p}v'_{\cal T}(g_{N,p}).
\end{equation}
Consequently,
\begin{eqnarray*}
|\Pi_{h}v_{\cal T}|_{1,\cal T'}^2&=&\sum_{i=1}^{N}\sum_{j=1}^p h_i^{-1}[v_{i,j}]^2
=\sum_{i=1}^{N}\sum_{j=1}^p h_i^{-1} \left(A_{i,j}\beta v'_{\cal T}(g_{i,j})\right)^2.
\end{eqnarray*}
  Noticing that $A_{i,j}\thicksim h_i$, we get
\begin{eqnarray}\label{equ-norm}
|\Pi_{h}v_{\cal T}|_{1,\cal T'}\thicksim |v_{\cal T}|_{G} \thicksim |v_{\cal T}|_{1}.
\end{eqnarray}
  Then the second inequality of \eqref{equi:norm} follows.
\end{proof}

We are now ready to present the inf-sup condition and the continuity of $a(\cdot,\cdot)$.

\begin{theorem}
  For all $u\in H^1, v_{\cal T'}\in V_{\cal T'}$, there holds
\begin{equation}\label{conti}
   a(u,v_{\cal T'})\le M \|u\|_{G}\|v_{\cal T'}\|_{\cal T'}.
\end{equation}
  Moreover, if the mesh size $h$ is sufficiently small, then
  \begin{equation}\label{infsup}
\inf_{v_{\cal T}\in U_{\cal T}}\sup_{w_{T'} \in V_{\cal T'} }
\frac{|a (v_{\cal T} ,w_{\cal T'} )|}{\|v_{\cal T} \|_{G} \|w_{\cal T'}\|_{\cal T'}}\ge c_0,
\end{equation}
where both $M, c_0$ are constants independent of the mesh-size $h$.
Consequently,
\begin{equation}\label{totalestimate}
   \|u-u_{\cal T}\|_G\le \frac{M}{c_0}\inf_{v_{\cal T}\in U_{\cal T}}\|u-v_{\cal T}\|_{G}.
\end{equation}
\end{theorem}
\begin{proof} By \eqref{bilinear1} and the Cauchy-Schwartz inequality, we have
\begin{eqnarray*}
   a(u,v_{\cal T'})
   &\le& |u|_{G}\left(\sum_{i=1}^N\sum_{j=1}^p\frac{\beta}{A_{i,j}}[v_{i,j}]^2\right)^{\frac 12}+\max(|\gamma|,|c|)\|u\|_{1}\left(\sum_{i=1}^N\sum_{j=1}^p h_iv_{i,j}^2\right)^{\frac 12}\\
      &\le & M\|u\|_G\|v_{\cal T'}\|_{\cal T'},
\end{eqnarray*}
   where the constant $M$ only depends on $\beta, \gamma, c$. Then \eqref{conti} follows.

   Recall the definition of the linear mapping $\Pi_h$, then we have
\[
a(v_{\cal T},\Pi_{h}v_{\cal T})=I_1+I_2,\ \ \forall v_{\cal T}\in U_{\cal T}
\]
  with
\[
I_1=\sum_{i=1}^{N}\sum_{j=1}^{p} [v_{i,j}]\beta(g_{i,j})v'_{\cal T}(g_{i,j}),\quad I_2=\sum_{i=1}^{N}\sum_{j=1}^{p_i} v_{i,j}\int_{g_{i,j}}^{g_{i,j+1}}\big(\gamma v'_{\cal T}(x)+cv_{\cal T}(x)\big)dx .
\]
   In light of \eqref{eq:2}, we have
\begin{eqnarray*}
I_1=\sum_{i=1}^{N}\sum_{j=1}^{p} A_{i,j}(\beta v'_{\cal T})^2(g_{i,j})
\ge
\beta_0|v_{\cal T}|_{1}^2.
\end{eqnarray*}
  To estimate $I_2$, we let $V(x)=\int_a^x \left( \gamma v'_{\cal T}(s)+cv_{\cal T}(s) \right) ds $ and
  denote by
\begin{eqnarray*}
    E_{i}&=&\int_{x_{i-1}}^{x_{i}}w(x)\beta(x) v'_{\cal T}(x)V(x)dx-\sum_{j=1}^{p} A_{i,j} (\beta v'_{\cal T})(g_{i,j})V(g_{i,j}),
\end{eqnarray*}
the error of Gauss quadrature in the interval $\tau_i$, $i\in\bZ_N$.
Then
\begin{eqnarray*}
I_2&=&- \sum_{i=1}^{N}\sum_{j=1}^{p}  [v_{i,j}] V(g_{i,j})= - \int_a^b w(x)\beta(x)v'_{\cal T}(x)V(x)dx + \sum_{i=1}^N E_i\\
&=& \int_a^b \left(\gamma v'_{\cal T}+cv_{\cal T} \right)v_{\cal T}(x) dx+\sum_{i=1}^{N} E_{i}
=\int_a^b cv^2_{\cal T}(x) dx+\sum_{i=1}^{N} E_{i},
\end{eqnarray*}
  where in the second and last steps, we have used the integration by parts and the fact that
    $v_{\cal T}(a)=v_{\cal T}(b)=0$.
On the other hand, the error of Gauss quadrature can be represented as (see, e.g.,
\cite{DavisRabinowitz1984}, p98, (2.7.12)))
\[
E_{i}=\frac{h_i^{2p+1}(p!)^4}{(2p+1)[(2p)!]^3}(\beta v'_{\cal T}V)^{(2p)}(\xi_i),
\]
where $\xi_i\in \tau_i$.
By the  Leibnitz formula of derivatives, we have
\begin{equation*}
    \left|(\beta v'_{\cal T}V)^{(2p)}(\xi_i)\right|\le \sum_{k=p+1}^{2p} \binom{2p}{k}\left| (\gamma v'_{\cal T}+c v_{\cal T})^{(k-1)} (\beta v'_{\cal T})^{(2p-k)}(\xi_i)\right| \le c_1 \|v_{\cal T}\|_{p,\infty,\tau_i}^2
\end{equation*}
with
\[
   c_1=\max\{\beta,\gamma, c\}\sum_{k=p+1}^{2p} \binom{2p}{k}.
\]
   Noticing that $\beta v_{\cal T}^{(k)}\in\mathbb P_p, k\in\bZ_p$,  the inverse inequality holds and thus
\[
  \|\beta v_{\cal T}\|_{p,\infty,\tau_i} \lesssim h_i^{-(p-\frac12)}|\beta v_{\cal T}|_{1,\tau_i},\ \ p\ge 1.
\]
   Then
\begin{equation*}\label{quad err}
|E_{i}| \leq \frac{c_1(p!)^4}{(2p+1)[(2p)!]^3} h_i^2 |\beta v_{\cal T}|_{1,\tau_i}^2.
\end{equation*}
  Plugging the estimate for $E_i$ into the formula of $I_2$ yields
\[
    I_2 \geq c \|v_{\cal T}\|^2_{0}-\frac{c_1(p!)^4}{(2p+1)[(2p)!]^3}h^2\left|v_{\cal T}\right|^2_{1}.
\]
Then for sufficiently small $h$, we have
\begin{eqnarray}\label{coer:1}
    a(v_{\cal T},\Pi_{h}v_{\cal T}) &\ge& \frac{\beta_0}{2}|v_{\cal T}|_{1}^2 +\frac{c}{2}\|v_{\cal T}\|_{0}^2\ge \frac12\min\{\beta_0,c\}\|v_{\cal T}\|_{1}^2.
\end{eqnarray}
  In light of \eqref{equi:norm}-\eqref{equ-norm},  there holds  for any $v_{\cal T}\in U_{\cal T}$,
\[
\sup_{w_{\cal T'}\in V_{\cal T'}} \frac{a(v_{\cal T},w_{\cal T'})}{\|w_{\cal T'}\|_{\cal T'}}
\ge \frac{a(v_{\cal T},\Pi_hv_{\cal T})}{\|\Pi_{h}v_{\cal T}\|_{\cal T'}}\ge c_0 \|v_{\cal T}\|_{G},
\]
where $c_0$  is a constant independent of the mesh size $h$.
The inf-sup condition \eqref{infsup} then follows.
Combining the continuity \eqref{conti}, inf-sup condition \eqref{infsup}, and the orthogonality of IFVM,  we derive \eqref{totalestimate}
 following similar arguments as in \cite{BabuskaAziz1972} or \cite{XuZikatanov2003}.
\end{proof}

\begin{remark}
    As we may observe in the proof of the above theorem,  \eqref{infsup} always holds no matter where the interface is. 
    In other words, the inf-sup condition of the IFVM is independent of the location of the interface point. However, the error bound $\frac{M}{c_0}$ in \eqref{totalestimate} 
    is dependent on the ratio $\rho=\frac{\beta_{\max}}{\beta_{\min}}$. 
\end{remark}

A direct consequence of the above theorem is the following error estimate for the IFVM.

\begin{corollary}\label{coro:11}
     Let $\mathcal{T} = \{\tau_i\}_{i=1}^N$ be a partition of $\Omega$ such that the interface $\alpha\in \tau_k$.  Let $u_{\cal T}\in  U_{\cal T}$ be the IFV solution of \eqref{eq: IFE method}, and $u\in \tilde W_\beta^{p+1,\infty}(\Omega)$ be the exact solution of \eqref{eq: DE} - \eqref{eq: jump condition}. Then there exists a constant $C$, depending 
     on $\rho=\frac{\beta_{\max}}{\beta_{\min}}$, $\gamma$, $c$ and $p$,  such that 
\begin{equation}\label{optimal}
   |u-u_{\cal T}|_1\le C h^{p}\|u\|_{p+1,\infty}. 
\end{equation}
\end{corollary}
\begin{proof}
Noticing that $\|\cdot\|_1\le \|\cdot\|_{G}$, we have from \eqref{totalestimate}
\[
   \|u-u_{\cal T}\|_1\le  \|u-u_{\cal T}\|_G\le \frac{M}{c_0}\inf_{v_{\cal T}\in U_{\cal T}}\|u-v_{\cal T}\|_G\le \frac{M}{c_0}\|u-u_I\|_G,
\]
  where $u_I$ is some interpolation function of $u$. Then \eqref{optimal} follows from the approximation theory of the immersed finite element space \cite{2009AdjeridLin}.
\end{proof}

\section{Superconvergence analysis}
In this section, we derive some superconvergence properties of IFVM. First we introduce a special Guass-Lobatto projection, which is of great importance in the superconvergence analysis.  For any $u\in \tilde W_{\beta}^{m,q}(\Omega), m\ge 1$, we have the following (generalized) Lobatto expansion of $u$ on each element $\tau_i$ \cite{2017CaoZhangZhang}:
\begin{equation}\label{eq: expansion on noninterface elem}
  u(x) |_{\tau_i}= \sum_{n = 0}^\infty u_{i,n}\phi_{i,n}(x),
\end{equation}
where
\[
   u_{i,0} = u(x_{i-1}), \ \  u_{i,1} = u(x_{i}),\ \ u_{i,n}=\frac{\displaystyle\int_{\tau_i} \beta u'(x)\phi_{i,n}'(x)dx}{\displaystyle\int_{\tau_i} \beta\phi_{i,n}'(x)\phi_{i,n}'(x)dx}.
\]
  We define the  Gauss-Lobatto projection $\mathcal{I}_h: \tilde W_{\beta}^{m,q}(\Omega) \to  U_{\mathcal{T}}$ as follows
\begin{equation}\label{eq: interpolation}
  (\mathcal{I}_hu)|_{\tau_i} =
      \sum\limits_{n=0}^p u_{i,n}\phi_{i,n}(x).
\end{equation}
%
Let $ \tilde U_{\cal T}= \{v\in C(\Omega): v|_{\tau_i} \in \text{span}\{\phi_{i,n}: n = 0,1,\cdots, p\}, v(a)=0\}$. Then we  define a special function $\omega_{\cal T}\in \tilde U_{\cal T}$ as follows.
\begin{equation}\label{corr:w}
    \beta\omega'_{\cal T}(g_{i,j})=\beta(u-{\cal I}_h u)'(g_{i,j})-\gamma (u-{\cal I}_hu)(g_{i,j}),\ \ (i,j)\in\bZ_{N}\times\bZ_p.
\end{equation}

\begin{lemma}  Let $u\in \tilde W^{2p+1,\infty}_{\beta}(\Omega)$ and  $\omega_{\cal T}\in   \tilde U_{\cal T}$ be the special function defined by \eqref{corr:w}.  Then $\omega_{\cal T}$ is well-defined, and
for all $p\ge 2$
 \begin{equation}\label{esti:w}
    \|\omega_{\cal T}\|_{0,\infty}\le C h^{p+2}\|u\|_{2p+1,\infty},
\end{equation}
 where $C$ is a positive constant  dependent on the coefficients $\beta$ and $\gamma$.
 \end{lemma}
\begin{proof} First, $\beta \omega'_{\cal T}\in\mathbb P_{p-1}$ is uniquely determined by the first condition of \eqref{corr:w} and thus $\omega'_{\cal T}$ is well-defined. Since $\omega_{\cal T}$ is continuous satisfying $\omega_{\cal T}(a)=0$, then $\omega_{\cal T}$ is uniquely determined.
  By the approximation property of ${\cal I}_h$ (see, \cite{2017CaoZhangZhang}), we get
\[
   \|u-{\cal I}_hu\|_{0,\infty}\lesssim h^{p+1}|u|_{p+1,\infty},\ \
   \beta(u-{\cal I}_hu)'(g_{i,j})\lesssim h^{p+1}|u|_{p+2,\infty},
\]
  which gives
\[
   \|\beta \omega'_{\cal T}\|_{0,\infty,\tau_i}\lesssim h^{p+1}\|u\|_{p+2,\infty}.
\]
  On the other hand, by Gauss quadrature,
\begin{eqnarray*}
    \omega_{\cal T}(x_i)-\omega_{\cal T}(x_{i-1})&=&\int_{\tau_i}\omega'_{\cal T}(x)dx=\sum_{j=1}^pA_{i,j}(\beta\omega'_{\cal T})(g_{i,j})\\
    &=&\sum_{j=1}^pA_{i,j}\left(\beta(u-{\cal I}_hu)'+\gamma(u-{\cal I}_hu)\right)(g_{i,j})\\
    &=&\int_{\tau_i}\frac{1}{\beta}\left(\beta(u-{\cal I}_hu)'+\gamma (u-{\cal I}_hu)\right)(x)dx-E_i,
\end{eqnarray*}
  where
\[
   E_i=\int_{\tau_i}\frac{1}{\beta}\left(\beta(u-{\cal I}_hu)'+\gamma (u-{\cal I}_hu)\right)(x)dx-\sum_{j=1}^pA_{i,j}\left(\beta(u-{\cal I}_hu)'+\gamma(u-{\cal I}_hu)\right)(g_{i,j})
\]
  denotes the error of Gauss quadrature in $\tau_i$. By the orthogonality of the Lobotto polynomials,
  we have $(u-{\cal I}_hu)\bot\mathbb P_0(\tau_i), i\neq k$, then
\[
  \omega_{\cal T}(x_i)-\omega_{\cal T}(x_{i-1}) =
  \left\{
    \begin{array}{ll}
      -E_i, & \text{if}~ i\neq k, \vspace{1mm}\\
      \int_{\tau_k}\frac{\gamma}{\beta}(u-{\cal I}_hu)(x)dx-E_k, & \text{if}~ i= k.
    \end{array}
  \right.
\]
    Noticing that
\[
E_{i}=\frac{h_i^{2p+1}(p!)^4}{(2p+1)[(2p)!]^3}(\beta (u-{\cal I}_hu)'+\gamma (u-{\cal I}_hu))^{(2p)}(\xi_i),\ \ \xi_i\in\tau_i,
\]
  we have
\[
   |E_i|\lesssim h^{2p+1}\|u\|_{2p+1,\infty,\tau_i},
\]
   which yields
\begin{eqnarray*}
|\omega_{\cal T}(x_i)-\omega_{\cal T}(x_{i-1})|\lesssim h^{2p+1}\|u\|_{2p+1,\infty},&~~& i\neq k,  \\
|\omega_{\cal T}(x_k)-\omega_{\cal T}(x_{k-1})|\lesssim h^{p+2}\|u\|_{2p+1,\infty}.
\end{eqnarray*}
  Using the fact $\omega_{\cal T}(a)=\omega_{\cal T}(x_0)=0$, we have for all $i\in\bZ_N$
\[
    |\omega_{\cal T}(x_i)|\lesssim h^{p+2}\|u\|_{2p+1,\infty},\ \ p\ge 2.
\]
 Then for all $x\in\tau_i$,
\[
   |\omega_{\cal T}(x)|=\left|\omega_{\cal T}(x_{i-1})+\int_{x_{i-1}}^{x}\omega'_{\cal T}(x) dx\right|\lesssim h^{p+2}\|u\|_{2p+1,\infty}.
\]
  This finishes our proof.
\end{proof}

    We define a linear interpolant of $\omega_{\cal T}$ on $[a,b]$ as follows.
\begin{equation}\label{interp:w}
    \omega_{I}(x)=\omega_{\cal T}(b)C_b\int_{a}^x\frac{1}{\beta(x)}dx,
\end{equation}
  where $C_b=(\frac{\alpha-a}{\beta^-}+\frac{b-\alpha}{\beta^+})^{-1}$.
 It is easy to check that
 \[
     \omega_{I}(a)=0=\omega_{\cal T}(a),\ \ \omega_{I}(b)=\omega_{\cal T}(b),\ \
     \bigjump{ \omega_I(\alpha)} = 0,\ \ \bigjump{\beta \omega_I^{(j)}(\hat\alpha)} = 0,\ \forall~ j = 1,2, \cdots, p.
\]
  Apparently, $\omega_I\in\tilde U_{\cal T}$ and $\omega_{\cal T}-\omega_{I}\in U_{\cal T}$. Moreover,
  there holds
\begin{equation}\label{est:wi}
    |\omega_I(x)|+|\beta\omega'_I(x)|\lesssim |C_b\omega_{\cal T}(b)|\lesssim \|\omega_{\cal T}\|_{0,\infty}\lesssim h^{p+2}\|u\|_{2p+1,\infty},\ \ \forall x\in\Omega.
\end{equation}

 Now we are ready to show our superconvergence properties of the IFVM.

 \begin{theorem}\label{theo:11}
Let $\mathcal{T} = \{\tau_i\}_{i=1}^N$ be an partition of $\Omega$ such that the interface $\alpha\in \tau_k$.  Let $u_{\cal T}\in U_{\cal T}$ be the IFV solution of \eqref{eq: IFE method} with $p\ge 2$, and $u\in \tilde W_\beta^{2p+1,\infty}(\Omega)$ be the exact solution of \eqref{eq: DE} - \eqref{eq: jump condition}. Then
\begin{itemize}
   \item The IFV solution $u_{\cal T}$ is superclose to the Gauss-Lobatto projection of the exact solution, i.e.,
\begin{equation}\label{superclose}
   \|u_{\cal T}-{\cal I}_hu\|_{0,\infty}=O(h^{p+2})
\end{equation}
   \item The function value approximation of $u_{{\cal T}}$ is superconvergent at roots of $\phi_{i,p+1}$, with an order of $p+2$. That is,
   \begin{equation}\label{super:lob}
        (u-u_{\cal T})(l_{i,j})=O(h^{p+2}),
   \end{equation}
     where $l_{i,j}$ are zeros of $\phi_{i,p+1}$.

   \item The flux approximation of $\beta u_{\cal T}'$ is superconvergent with an order
   of $p+1$
   at the Gauss points $g_{i,j}, (i,j)\in\bZ_N\times\bZ_p$, \emph{i.e.},
\begin{equation}\label{super:gau}
    \beta(u-u_{\cal T})'(g_{i,j}) =O(h^{p+1}).
\end{equation}
 \item For diffusion only equation, i.e., $\gamma=c=0$, there hold
\begin{eqnarray}\label{sup:1}
   &&\beta(u-u_{\cal T})'(g_{i,j})=O(h^{2p}),\ \ \ (u-u_{\cal T})(x_i)=O(h^{2p}),\\\label{sup:2}
   &&(u-u_{\cal T})(x_i)-(u-u_{\cal T})(x_{i-1})=O(h^{2p+1}).
\end{eqnarray}
\end{itemize}
  Here  the hidden constants are dependent on the ratio $\rho=\frac{\beta_{\max}}{\beta_{\min}}$, $\gamma$, $c$ and $p$. 
\end{theorem}
\begin{proof}  First,  let
\[
   u_I={\cal I}_hu+\omega_{\cal T}-\omega_I,
\]
  where $\omega_{\cal T}$ is defined by \eqref{corr:w}, and $\omega_I$ is the linear interpolant of $\omega_{\cal T}$ given by \eqref{interp:w}, and  define a operator $D_x^{-1}$ on all $v\in H^1(\Omega)$,
\[
   D_x^{-1}v(x)=\int_{a}^xv(x)dx.
\]
  For all $v_{\cal T'}\in V_{\cal T'}$,
  it follows from \eqref{bilinear1}
\begin{eqnarray*}
   a(u-u_I,v_{\cal T'})&=&\sum_{i=1}^N\sum_{j=1}^p[v_{i,j}](\beta(u-u_I)'-\gamma (u-u_I)-cD_x^{-1}(u-u_I))(g_{i,j})\\
   &=&\sum_{i=1}^N\sum_{j=1}^p[v_{i,j}](\beta\omega_I'+\gamma (\omega_{\cal T}-\omega_I)-cD_x^{-1}(u-u_I))(g_{i,j}),
\end{eqnarray*}
  where in the last step, we have used  the definition of $\omega_{\cal T}$ in \eqref{corr:w}, which yields
\[
   (\beta(u-u_I)'-\gamma(u-u_I))(g_{i,j})=\gamma(\omega_{\cal T}-\omega_I)(g_{i,j})+\beta\omega_I'(g_{i,j}).
\]
 Noticing that $(u-{\cal I}_hu)\bot \mathbb P_0(\tau_i), i\neq k$, we have for all $x\in\tau_i$
\[
   D_x^{-1}(u-u_I)(x)=\left\{
   \begin{array}{ll}
       \int_{x_{i-1}}^x(u-{\cal I}_hu)(x)dx-\int_{a}^x(\omega_{\cal T}-\omega_{I})(x)dx, \ i\le k,\\
       \int_{x_{k-1}}^{x_k}(u-{\cal I}_hu)(x)dx+\int_{x_{i-1}}^x(u-{\cal I}_hu)(x)dx-\int_{a}^x(\omega_{\cal T}-\omega_{I})(x)dx,\ i> k,
   \end{array}
   \right.
\]
  which yields, together with \eqref{esti:w} and \eqref{est:wi}
\[
   \|D_x^{-1}(u-u_I)\|_{0,\infty}\lesssim h\|u-{\cal I}_hu\|_{0,\infty}+\|\omega_{\cal T}\|_{0,\infty}\lesssim
   h^{p+2}\|u\|_{2p+1,\infty}.
\]
  Then by the Cauchy-Schwartz inequality, \eqref{esti:w} and \eqref{est:wi}
\begin{eqnarray*}
  \left|a(u-u_I,v_{\cal T'})\right|
  &\lesssim&|v_{\cal T'}|_{1,\cal T'}\left(\sum_{i=1}^N\sum_{j=1}^pA_{i,j}(\beta\omega_I'+\gamma (\omega_{\cal T}-\omega_I)-cD_x^{-1}(u-u_I))^2(g_{i,j})\right)^{\frac 12}\\
  &\lesssim &|v_{\cal T'}|_{1,\cal T'}\left(\|\beta\omega'_{I}\|_{0,\infty}+\|\omega_{\cal T}-\omega_I\|_{0,\infty}+\|D_x^{-1}(u-{u_I})\|_{0,\infty}\right)\\
  &\lesssim& h^{p+2}\|u\|_{2p+1,\infty}|v_{\cal T'}|_{1,\cal T'},\ \ \forall v_{\cal T'}\in V_{\cal T'}.
\end{eqnarray*}
 %
 %
  Now we choose $v_{\cal T}=u_I-u_{\cal T}\in U_{\cal T}$ in \eqref{infsup} and use the orthogonality to obtain
\begin{eqnarray*}
  \|u_h-u_I\|_{1}\le \|u_h-u_I\|_{G}&\le& \frac{1}{c_0}\sup_{v_{\cal T'}\in V_{\cal T'}} \frac{a(u_h-u_I,v_{\cal T'})}{\|v_{\cal T'}\|_{\cal T'}}\lesssim  h^{p+2}\|u\|_{2p+1,\infty}.
\end{eqnarray*}
    Noticing  that $(u_h-u_I)(a)=0$,  we have
 \[
     (u_h-u_I)(x)=\int_{a}^x  (u_h-u_I)'(x) dx,
 \]
   which yields
 \[
      \|u_h-u_I\|_{0,\infty}\lesssim  |u_h-u_I|_1\lesssim h^{p+2}\|u\|_{2p+1,\infty},
 \]
   and thus,
  \[
      \|u_h-{\cal I}_hu\|_{0,\infty}\le \|u_h-u_I\|_{0,\infty}+\|\omega_{\cal T}-\omega_I\|_{0,\infty}\lesssim h^{p+2}\|u\|_{2p+1,\infty}.
 \]
    This finishes the proof of \eqref{superclose}. Since $\beta(u_{\cal T}-{\cal I}_hu)'\in\mathbb P_{p-1}$, the inverse inequality holds. Then
  \[
      \|\beta (u_{\cal T}-{\cal I}_hu)'\|_{0,\infty}\lesssim  h^{-1}  \|\beta (u_{\cal T}-{\cal I}_hu)\|_{0,\infty}\lesssim h^{p+1}\|u\|_{2p+1,\infty}.
 \]
     It has been proved in \cite{2017CaoZhangZhang} that
\[
    (u-{\cal I}_hu)(l_{i,j})\lesssim h^{p+2}\|u\|_{p+2,\infty},\ \
    \beta(u-{\cal I}_hu)'(g_{i,j})\lesssim h^{p+1}\|u\|_{p+2,\infty}.
\]
    Then \eqref{super:lob}- \eqref{super:gau} follow from the triangle inequality.

   Now we consider the special case $\gamma=c=0$. For simplicity, we denote $e_u=u-u_{\cal T}$.
   It follows from the FV scheme \eqref{conserve} that
\[
   \beta e_u'(g_{i,j})-\beta e_u'(g_{i,j+1})=0.
\]
   In other words,
\[
 \beta e_u'(g_{i,j+1})=C_0,
\]
where $C_0$ is a constant.    Summing up all $(i,j)$ yields
\[
   C_0\sum_{i=1}^N\sum_{j=1}^pA_{i,j}=\sum_{i=1}^N\sum_{j=1}^p A_{i,j}\beta e_u'(g_{i,j})=\int_{a}^be_u'(x)dx-\sum_{i=1}^NE_i=-\sum_{i=1}^NE_i,
\]
  where the error of Gauss quadrature $E_i$ in each element $\tau_i$ can be represented as
\[
   |E_i|=\frac{h_i^{2p+1}(p!)^4}{(2p+1)[(2p)!]^3}|e_u^{(2p+1)}(\xi_i)|\lesssim h^{2p+1}\|u\|_{2p+1,\infty},
\]
  where $\xi_i\in\tau_i$ is some point. Noticing that $\sum_{i=1}^N\sum_{j=1}^pA_{i,j} \thicksim(b-a)$,  we have
\[
  |C_0|\lesssim \frac{1}{b-a}\sum_{i=1}^N|E_i|\lesssim h^{2p}\|u\|_{2p+1,\infty},
\]
  and thus
\[
     |\beta e_u'(g_{i,j+1})|=|C_0|\lesssim h^{2p}\|u\|_{2p+1,\infty}.
\]
   Again, we use Gauss quadrature to obtain
\[
   e_u(x_i)-e_u(x_{i-1})=\int_{\tau_i} e'_u(x)dx=\sum_{j=1}^pA_{i,j}\beta e'_u(g_{i,j})+E_i=h_iC_0+E_i,
\]
  and thus
\[
   e_u(x_j)=e_u(x_0)+C_0\sum_{i=1}^{j}h_i+\sum_{i=1}^{j}E_i=C_0\sum_{i=1}^{j}h_i+\sum_{i=1}^{j}E_i.
\]
    Combining the estimates for $C_0$ and $E_i$, the desired results \eqref{sup:1}-\eqref{sup:2} follow. The proof is complete.
\end{proof}
\begin{remark}
   As a direct consequence of \eqref{superclose}, we immediately obtain the optimal convergence rate of the IFV solution under the $L^2$ norm. That is
  \[
     \|u-u_{\cal T}\|_0\le \|u-{\cal I}_hu\|_0+ \|{\cal I}_hu-u_{\cal T}\|_0=O(h^{p+1}).
  \]
\end{remark}
\begin{remark}
   The error estimate \eqref{optimal} and the superconvergence results \eqref{superclose} - \eqref{sup:2} can be readily extended to interface problems with multiple discontinuity.
\end{remark}

\begin{remark}
   In general, there is no superconvergence behavior on the interface point $\alpha$, unless it coincides with the generalized Gauss or Lobatto points. However, if the interfacecoincides with a mesh point, the IFVM becomes the standard FVM, and the function value is superconvergent of order $O(h^{2k})$ according to the analysis in \cite{Cao;Zhang;Zou2012}
\end{remark}

  \begin{remark}
           The error estimate \eqref{optimal} and the superconvergence results \eqref{superclose}-\eqref{super:gau} are valid for smooth variable coefficients $\gamma=\gamma(x)$ and 
          $c=c(x)$, e.g., $\gamma, c\in C^1(\Omega)$. This can be proved using the same argument as for the constant coefficients $\gamma$ and $c$, 
  \end{remark}

 \begin{remark}
    The regularity assumption   $u\in \tilde W_\beta^{2p+1,\infty}(\Omega)$  in Theorem \ref{theo:11}  is stronger than that for the counterpart IFEM in \cite{2017CaoZhangZhang}, which is $u\in \tilde W_\beta^{p+2,\infty}(\Omega)$.  
   As we may observe in our analysis, the regularity assumption  on the jump condition 
   \eqref{eq: space2}  for high order scheme is necessary. In other words, if the exact solution only satisfies the jump condition
    \eqref{eq: jump condition} instead of \eqref{eq: space2} for $m> 1$, then even the optimal convergence rate will be impaired, and this is further demonstrated in our numerical experiments (Example 5.3). 
      \end{remark}



\section{Numerical Examples}
In this section, we present some numerical experiments to demonstrate the features of IFVM.

We test the same example as in \cite{2017CaoZhangZhang}. The exact solution is chosen as
\begin{equation}\label{eq: IFE ex1}
  u(x) =
  \left\{
    \begin{array}{ll}
      \dfrac{1}{\beta^-}\cos(x), & \text{if}~~x \in [0,\alpha), \\
      \dfrac{1}{\beta^+}\cos(x) + \left(\dfrac{1}{\beta^-} - \dfrac{1}{\beta^+}\right)\cos(\alpha), & \text{if}~~x \in (\alpha,1], \\
    \end{array}
  \right.
\end{equation}
where $\alpha = \pi/6$ is the interface point, and $(\beta^-,\beta^+) = (1,5)$ represents a moderate discontinuity of the diffusion coefficient.

We use a family of uniform meshes $\{\mathcal{T}_h\}, h>0$ where $h$ denotes the mesh size. We test the IFVM for polynomial degrees $p=1,2,3$. Due to the finite machine precision, we choose different sets of meshes for different polynomial degrees $p$. The convergence rate is calculated using linear regression of the errors. Error $e_{\cal T}=u_{\cal T}-u$ in the following norms will be calculated.
\begin{eqnarray*}
  &\|e_{\cal T}\|_{N} = \max\limits_{x\in{\{x_i\}}}|u_{\cal T}(x)-u(x)|,~~~~
  &\|e_{\cal T}\|_{0,\infty} = \max\limits_{x\in\Omega}|u_{\cal T}(x)-u(x)|, \\
  &\|e_{\cal T}\|_{L} = \max\limits_{x\in\{l_{ip}\}}|u_{\cal T}(x)-u(x)|,~~~
  &\|\beta e_{\cal T}'\|_{G} = \max\limits_{x\in\{g_{ip}\}}|\beta u_{\cal T}'(x)-\beta u'(x)|, \\
  &\|e_{\cal T}\|_{0} = \left(\displaystyle\int_\Omega|u_{\cal T}-u|^2dx\right)^{\frac{1}{2}},
  &|e_{\cal T}|_{1} = \left(\int_\Omega|u_{\cal T}'-u'|^2dx\right)^{\frac{1}{2}}, \\
  &\|e_{\cal T}\|_{P} = \max\limits_{i}|e_{\cal T}(x_i)-e_{\cal T}(x_{i-1})|.
\end{eqnarray*}
Here, $\|e_{\cal T}\|_{N}$ denotes the maximum error over all the nodes (mesh points). $\|e_{\cal T}\|_{0,\infty}$ is the infinity norm over the whole domain $\Omega$. This is computed by choosing 10 uniformly distributed points on each non-interface element, and 10 uniformly distributed points in each sub-element of an interface element, and then calculating the largest discrepancy.  $\|\beta e_{\cal T}'\|_{G}$ is the maximum error of flux over all (generalized) Gauss points. $\|e_{\cal T}\|_{L}$ is maximum solution error over all (generalized) Lobatto points. $\|e_{\cal T}\|_{0}$ and $|e_{\cal T}|_{1}$ are the standard Sobolev $L^2$- and semi-$H^1$- norms. $\|e_{\cal T}\|_{P}$ measures the maximum of the difference of errors at two consecutive nodes.

\begin{example} (diffusion interface problem)\end{example}
In this example, we test IFVM for the diffusion interface problem, i.e., $\gamma = c = 0$. Errors and convergence rates for linear, quadratic and cubic IFVM solutions are listed in Tables
\ref{table: general elliptic P1-IFV-1-5 diff}, \ref{table: general elliptic P2-IFV-1-5 diff}, and \ref{table: general elliptic P3-IFV-1-5 diff}, respectively. 
The convergence rates are consistent with our theoretical analysis in Theorem 4.2. In particular, we note that for quadratic and cubic IFVM solution, the flux error at Gauss points are of order $O(h^{2p})$, which is higher than IFEM solution $O(h^{p+1})$ \cite{2017CaoZhangZhang}.

\begin{table}[thb]
\begin{center}
\caption{Error of $P_1$ IFVM Solution with $\beta=[1,5]$, $\alpha = \pi/6$, $\gamma=c=0$.}
\label{table: general elliptic P1-IFV-1-5 diff}
\begin{small}
\begin{tabular}{|r|c|c|c|c|c|c|}
\hline
$1/h$ & $\|e_{\cal T}\|_{N}$ & $\|e_{\cal T}\|_{0,\infty}$ & $\|\beta e_{\cal T}'\|_{G}$ &$\|e_{\cal T}\|_{0}$ & $|e_{\cal T}|_{1}$ & $\|e_{\cal T}\|_{P}$ \\
\hline
    8 & 3.41e-05 & 1.92e-03 & 2.11e-04 & 9.71e-04 & 2.51e-02  & 2.14e-05\\
  16 & 8.19e-06 & 4.81e-04 & 5.14e-05 & 2.42e-04 & 1.25e-02  & 2.89e-06\\
  32 & 2.05e-06 & 1.20e-04 & 1.29e-05 & 6.06e-05 & 6.26e-03  & 3.82e-07\\
  64 & 5.22e-07 & 3.01e-05 & 3.25e-06 & 1.52e-05 & 3.14e-03  & 4.95e-08\\
128 & 1.33e-07 & 7.53e-06 & 8.19e-07 & 3.82e-06 & 1.58e-03  & 6.31e-09\\
256 & 3.32e-08 & 1.88e-06 & 2.05e-07 & 9.56e-07 & 7.88e-04  & 7.95e-10\\
512 & 8.30e-09 & 4.71e-07 & 5.12e-08 & 2.40e-07 & 3.94e-04  & 9.96e-11\\
\hline
rate & 1.99        & 1.99        & 2.00        & 2.00         & 1.00  & 2.95 \\
\hline\end{tabular}
\end{small}
\end{center}
\end{table}

\begin{table}[thb]
\begin{center}
\caption{Error of $P_2$ IFVM Solution with $\beta=[1,5]$, $\alpha = \pi/6$, $\gamma=c=0$.}
\label{table: general elliptic P2-IFV-1-5 diff}
\begin{small}
\begin{tabular}{|r|c|c|c|c|c|c|c|}
\hline
$1/h$ & $\|e_{\cal T}\|_{N}$ & $\|e_{\cal T}\|_{0,\infty}$ & $\|e_{\cal T}\|_{L}$& $\|\beta e_{\cal T}'\|_{G}$ &$\|e_{\cal T}\|_{0}$ & $|e_{\cal T}|_{1}$& $\|e_{\cal T}\|_{P}$ \\
\hline
    8 & 2.80e-09 & 6.87e-06 & 2.10e-07 & 1.79e-08 & 2.51e-06 & 1.32e-04  & 1.80e-09\\
  16 & 1.80e-10 & 8.98e-07 & 1.32e-08 & 1.12e-09 & 3.18e-07 & 3.33e-05  & 6.32e-11\\
  24 & 3.55e-11 & 2.70e-07 & 2.61e-09 & 2.22e-10 & 9.46e-08 & 1.48e-05  & 8.63e-12\\
  32 & 1.11e-11 & 1.15e-07 & 8.27e-10 & 6.97e-11 & 3.97e-08 & 8.25e-06  & 2.07e-12\\
  40 & 4.62e-12 & 5.90e-08 & 3.39e-10 & 2.93e-11 & 2.07e-08 & 5.38e-06  & 6.90e-13\\
  48 & 2.26e-12 & 3.55e-08 & 1.63e-10 & 1.48e-11 & 1.21e-08 & 3.76e-06  & 2.82e-13\\
  56 & 1.27e-12 & 2.23e-08 & 8.82e-11 & 7.94e-12 & 7.57e-09 & 2.76e-06  & 1.35e-13\\
\hline
rate & 3.97        & 2.95        & 4.00        & 3.97        & 2.98         & 1.99   &  4.89\\
\hline
\end{tabular}
\end{small}
\end{center}
\end{table}

\begin{table}[thb]
\begin{center}
\caption{Error of $P_3$ IFVM Solution with $\beta=[1,5]$, $\alpha = \pi/6$, $\gamma=c=0$.}
\label{table: general elliptic P3-IFV-1-5 diff}
\begin{small}
\begin{tabular}{|r|c|c|c|c|c|c|c|}
\hline
$1/h$ & $\|e_{\cal T}\|_{N}$ & $\|e_{\cal T}\|_{0,\infty}$ & $\|e_{\cal T}\|_{L}$& $\|\beta e_{\cal T}'\|_{G}$ &$\|e_{\cal T}\|_{0}$ & $|e_{\cal T}|_{1}$& $\|e_{\cal T}\|_{P}$ \\
\hline
  4 & 6.00e-12 & 1.87e-06 & 7.29e-09 & 3.91e-11 & 8.96e-07 & 3.41e-05  & 6.00e-12 \\
  5 & 1.30e-12 & 7.68e-07 & 1.93e-09 & 9.53e-12 & 3.53e-07 & 1.69e-05  & 4.19e-12\\
  6 & 5.45e-13 & 3.71e-07 & 1.02e-09 & 3.51e-12 & 1.77e-08 & 1.01e-05  & 6.03e-13\\
  7 & 1.99e-13 & 2.01e-07 & 4.09e-10 & 1.31e-12 & 9.35e-08 & 6.23e-06  & 1.41e-13\\
  8 & 9.69e-14 & 1.18e-07 & 2.50e-10 & 6.26e-13 & 5.60e-08 & 4.27e-06  & 4.19e-14\\
  9 & 4.26e-14 & 7.34e-08 & 1.24e-10 & 3.18e-13 & 3.45e-08 & 2.95e-06  & 2.45e-14\\
\hline
rate & 5.97     & 3.99     & 4.88     & 5.92     & 4.00     & 3.00  & 6.70\\
\hline\end{tabular}
\end{small}
\end{center}
\end{table}

\begin{example} (General elliptic equations).\end{example}
In the example, we test the superconvergence behavior for general second-order equation, e.g., $\gamma=1$ and $c=1$. Tables
\ref{table: general elliptic P1-IFV-1-5} - \ref{table: general elliptic P3-IFV-1-5} report the errors and convergence rates of $P_1$, $P_2$, and $P_3$ IFVM approximation, respectively. Again, these data indicate the validity of our theoretical analysis. In Figures \ref{fig: error P1 IFV} - \ref{fig: error P3 IFV}, we plot the solution error and the flux error in a uniform mesh consists of eight elements. Note that the interface $\alpha = \pi/6$, depicted by an black dot, is in the fifth element. The (generalized) Lobatto points and the (generalized) Gauss points are show in red color. Clearly, we can see that solution errors and flux errors at these special points are much closer to zero, than the majority of the points. This again shows the superconvergence behavior of IFVM.

\begin{table}[thb]
\begin{center}
\caption{Error of $P_1$ IFVM Solution with $\beta=[1,5]$, $\alpha = \pi/6$, $\gamma=1$, $c=1$.}
\label{table: general elliptic P1-IFV-1-5}
\begin{small}
\begin{tabular}{|r|c|c|c|c|c|c|}
\hline
$1/h$ & $\|e_{\cal T}\|_{N}$ & $\|e_{\cal T}\|_{0,\infty}$ & $\|\beta e_{\cal T}'\|_{G}$ &$\|e_{\cal T}\|_{0}$ & $|e_{\cal T}|_{1}$ & $\|e_{\cal T}\|_{P}$ \\
\hline
    8 & 7.64e-05 & 1.92e-03 & 1.21e-03 & 9.98e-04 & 2.51e-02  & 5.49e-05\\
  16 & 2.03e-05 & 4.81e-04 & 3.05e-04 & 2.49e-04 & 1.25e-02  & 7.76e-06\\
  32 & 4.56e-06 & 1.20e-04 & 7.75e-05 & 6.22e-05 & 6.26e-03  & 9.70e-07\\
  64 & 1.17e-06 & 3.01e-05 & 1.95e-05 & 1.56e-05 & 3.14e-03  & 1.25e-07\\
128 & 2.81e-07 & 7.53e-06 & 4.91e-06 & 3.91e-06 & 1.58e-03  & 1.55e-08\\
256 & 7.02e-08 & 1.88e-06 & 1.23e-06 & 9.78e-07 & 7.88e-04  & 1.95e-09\\
512 & 1.76e-08 & 4.71e-07 & 3.07e-07 & 2.44e-07 & 3.94e-04  & 2.45e-10\\
\hline
rate & 2.02        & 1.99        & 1.99        & 2.00         & 1.00  & 2.97 \\
\hline\end{tabular}
\end{small}
\end{center}
\end{table}

\begin{table}[thb]
\begin{center}
\caption{Error of $P_2$ IFVM Solution with $\beta=[1,5]$, $\alpha = \pi/6$, $\gamma=1$, $c=1$.}
\label{table: general elliptic P2-IFV-1-5}
\begin{small}
\begin{tabular}{|r|c|c|c|c|c|c|c|}
\hline
$1/h$ & $\|e_{\cal T}\|_{N}$ & $\|e_{\cal T}\|_{0,\infty}$ & $\|e_{\cal T}\|_{L}$& $\|\beta e_{\cal T}'\|_{G}$ &$\|e_{\cal T}\|_{0}$ & $|e_{\cal T}|_{1}$& $\|e_{\cal T}\|_{P}$ \\
\hline
    8 & 5.46e-08 & 6.68e-06 & 1.71e-07 & 6.67e-06 & 2.51e-06 & 1.32e-04  & 2.61e-08\\
  16 & 8.84e-09 & 8.90e-07 & 1.23e-08 & 8.95e-06 & 3.18e-07 & 3.33e-05  & 1.39e-09\\
  24 & 1.84e-09 & 2.68e-07 & 2.49e-09 & 2.70e-07 & 9.46e-08 & 1.48e-05  & 1.90e-10\\
  32 & 2.97e-10 & 1.14e-07 & 7.92e-10 & 1.14e-07 & 3.97e-08 & 8.25e-06  & 3.20e-11\\
  40 & 4.62e-11 & 5.86e-08 & 3.25e-10 & 5.90e-08 & 2.07e-08 & 5.38e-06  & 6.69e-12\\
  48 & 3.32e-11 & 3.54e-08 & 1.58e-10 & 3.55e-08 & 1.21e-08 & 3.76e-06  & 3.27e-12\\
  56 & 4.92e-11 & 2.22e-08 & 8.63e-11 & 2.23e-08 & 7.57e-09 & 2.76e-06  & 2.42e-12\\
\hline
rate & 4.14        & 2.93        & 3.91        & 2.93        & 2.98         & 1.99   &  5.03\\
\hline
\end{tabular}
\end{small}
\end{center}
\end{table}


\begin{table}[thb]
\begin{center}
\caption{Error of $P_3$ IFVM Solution with $\beta=[1,5]$, $\alpha = \pi/6$, $\gamma=1$, $c=1$.}
\label{table: general elliptic P3-IFV-1-5}
\begin{small}
\begin{tabular}{|r|c|c|c|c|c|c|c|}
\hline
$1/h$ & $\|e_{\cal T}\|_{N}$ & $\|e_{\cal T}\|_{0,\infty}$ & $\|e_{\cal T}\|_{L}$& $\|\beta e_{\cal T}'\|_{G}$ &$\|e_{\cal T}\|_{0}$ & $|e_{\cal T}|_{1}$& $\|e_{\cal T}\|_{P}$ \\
\hline
  4 & 6.56e-09 & 1.89e-06 & 9.81e-08 & 2.02e-06 & 8.95e-07 & 3.41e-05  & 3.55e-09 \\
  6 & 1.82e-09 & 3.74e-07 & 1.29e-08 & 4.03e-07 & 1.77e-07 & 1.01e-05  & 7.17e-10\\
  8 & 6.56e-10 & 1.18e-07 & 3.30e-09 & 1.28e-07 & 5.60e-08 & 4.27e-06  & 2.01e-10\\
10 & 2.56e-10 & 4.85e-08 & 1.13e-09 & 5.24e-08 & 2.30e-08 & 2.19e-06  & 6.42e-11\\
12 & 9.88e-11 & 2.34e-08 & 4.52e-10 & 2.52e-08 & 1.11e-08 & 1.27e-06  & 2.09e-11\\
14 & 3.58e-11 & 1.26e-08 & 2.00e-10 & 1.36e-08 & 5.98e-09 & 7.95e-07  & 6.53e-12\\
16 & 1.20e-11 & 7.39e-09 & 9.70e-11 & 7.96e-09 & 3.50e-09 & 5.32e-07  & 1.93e-12\\
18 & 3.09e-12 & 4.61e-09 & 5.09e-11 & 4.96e-09 & 2.18e-09 & 3.73e-07  & 4.40e-13\\
\hline
rate & 4.88     & 4.00     & 4.99     & 4.00     & 4.00     & 3.00  & 5.77\\
\hline\end{tabular}
\end{small}
\end{center}
\end{table}

\begin{figure}[htb]
  \centering
  \includegraphics[width=.48\textwidth]{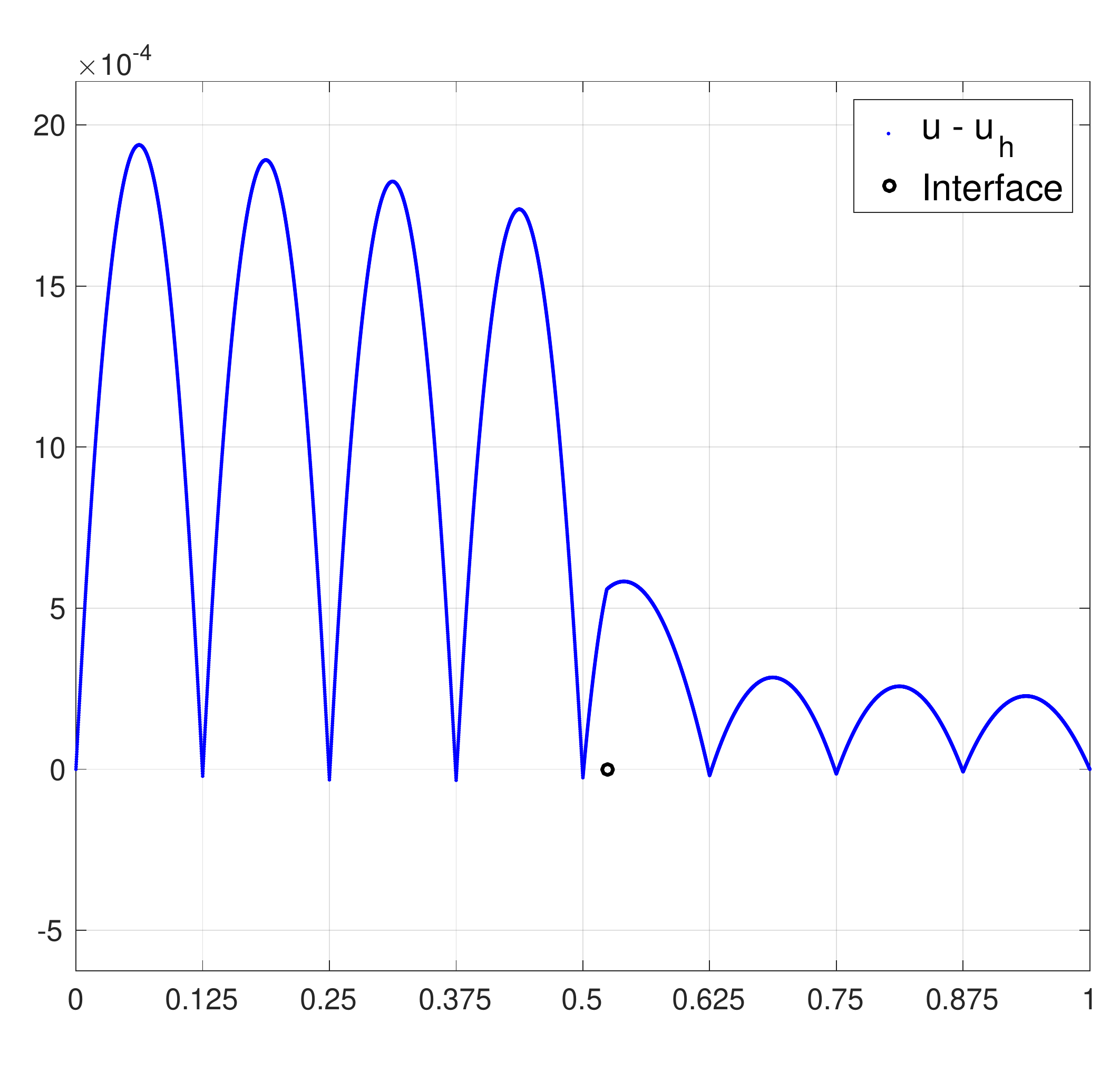}~
  \includegraphics[width=.48\textwidth]{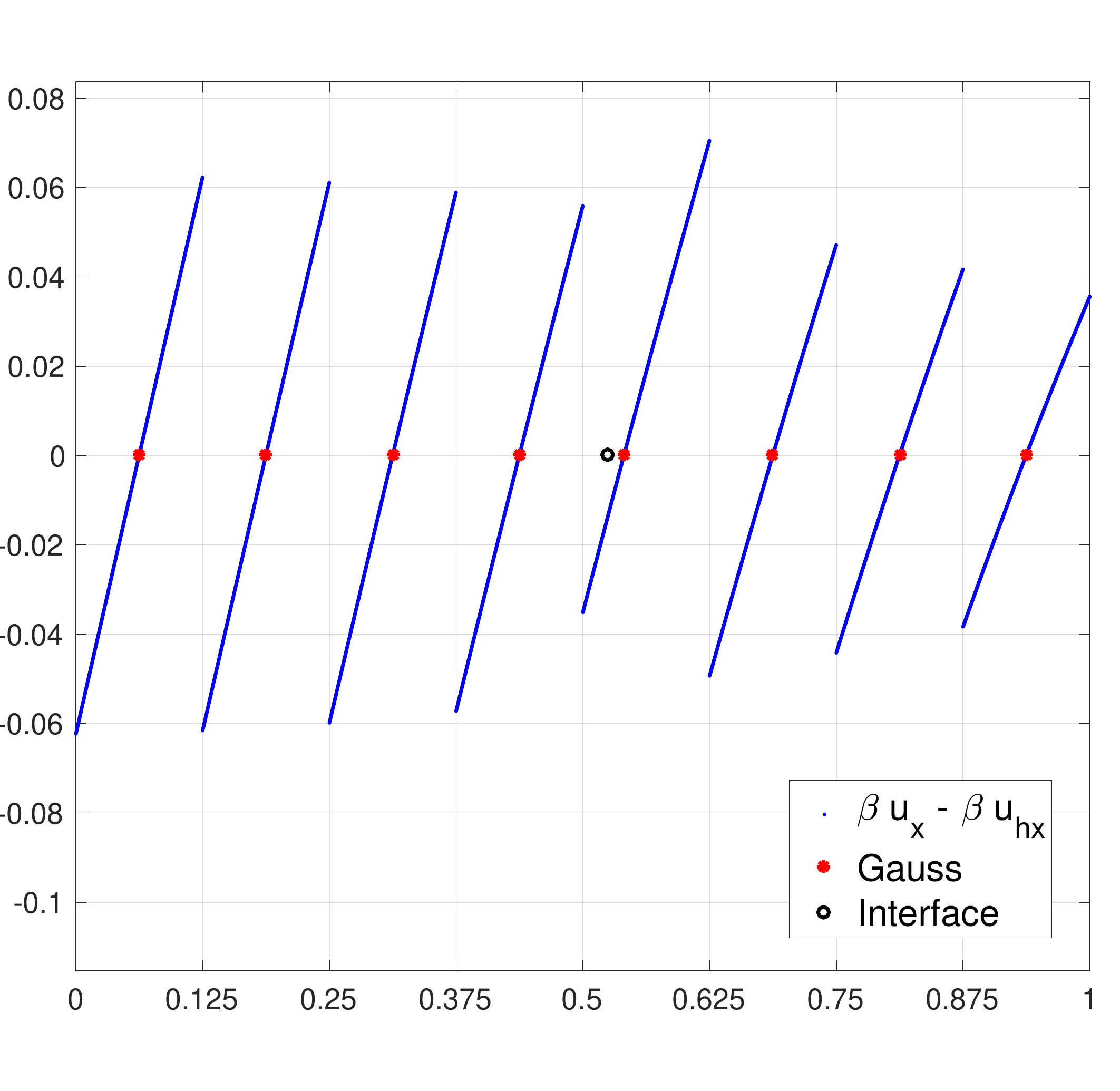}
  \caption{Error and flux error of $P_1$ IFVM solution. $\beta=\{1,5\}$, $\alpha = \dfrac{\pi}{6}$}
  \label{fig: error P1 IFV}
\end{figure}

\begin{figure}[htb]
  \centering
  \includegraphics[width=.48\textwidth]{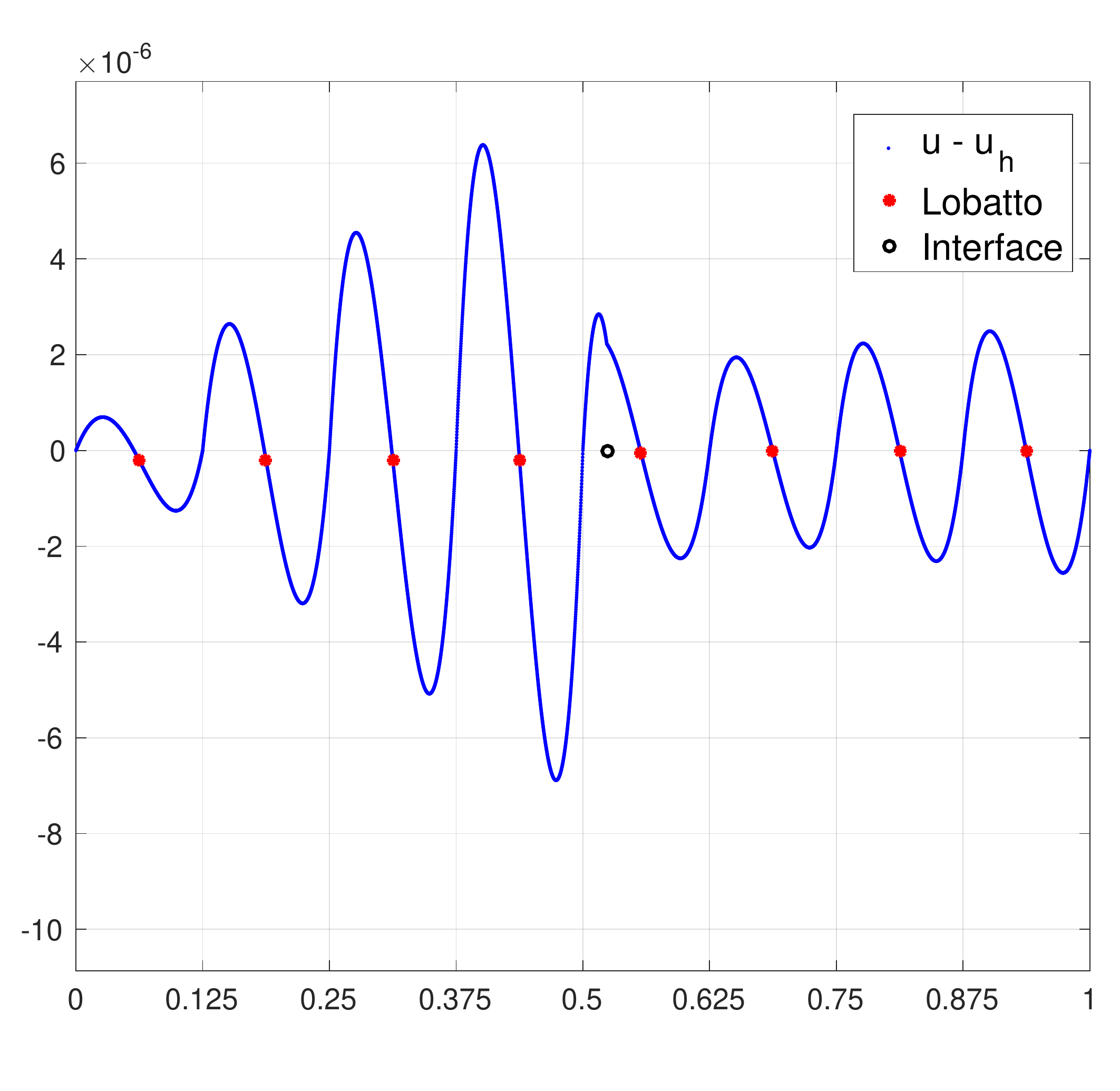}~
  \includegraphics[width=.48\textwidth]{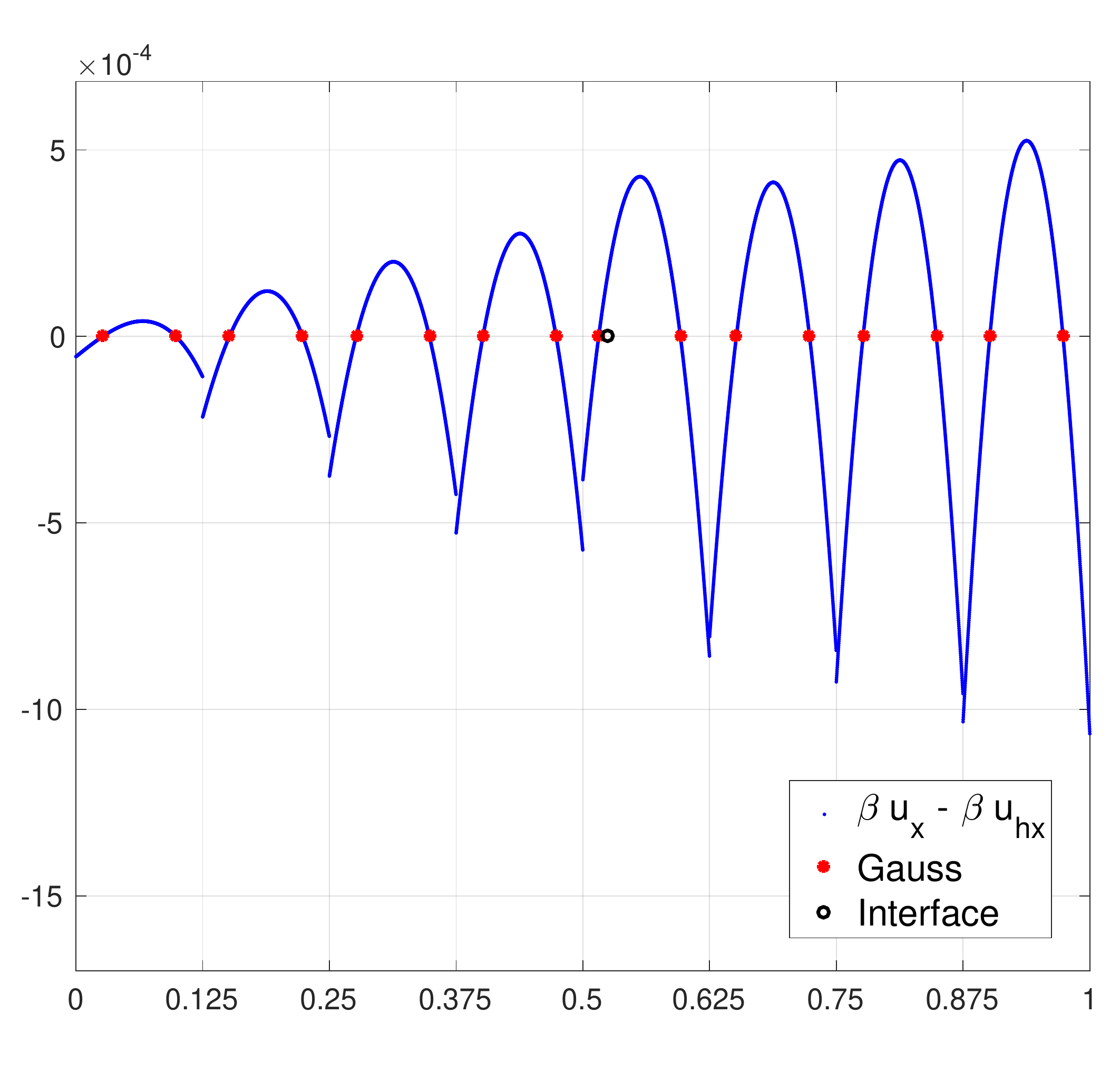}
  \caption{Error and flux error of $P_2$ IFVM solution. $\beta=\{1,5\}$, $\alpha = \dfrac{\pi}{6}$}
  \label{fig: error P2 IFV}
\end{figure}

\begin{figure}[htb]
  \centering
  \includegraphics[width=.48\textwidth]{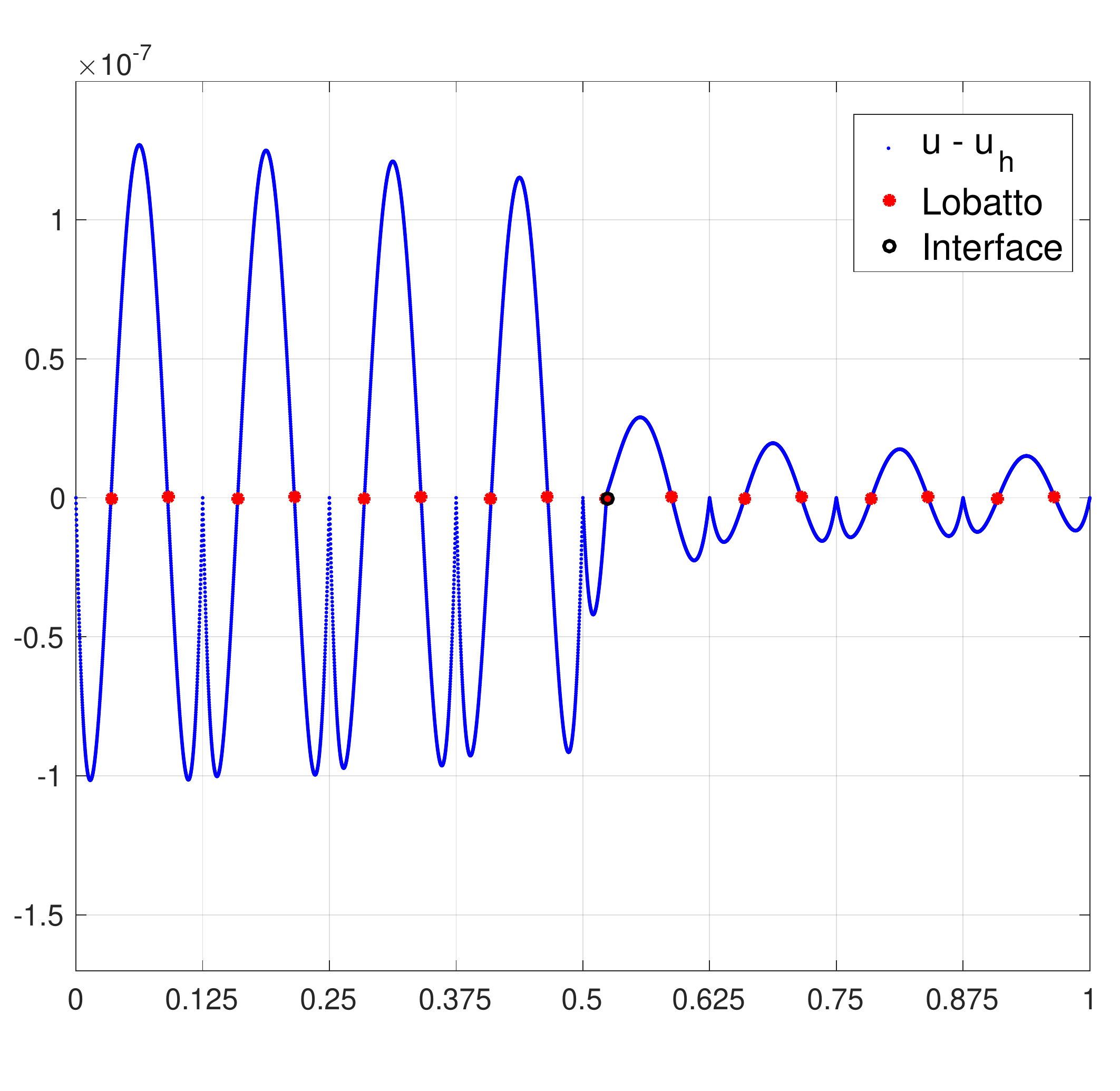}~
  \includegraphics[width=.48\textwidth]{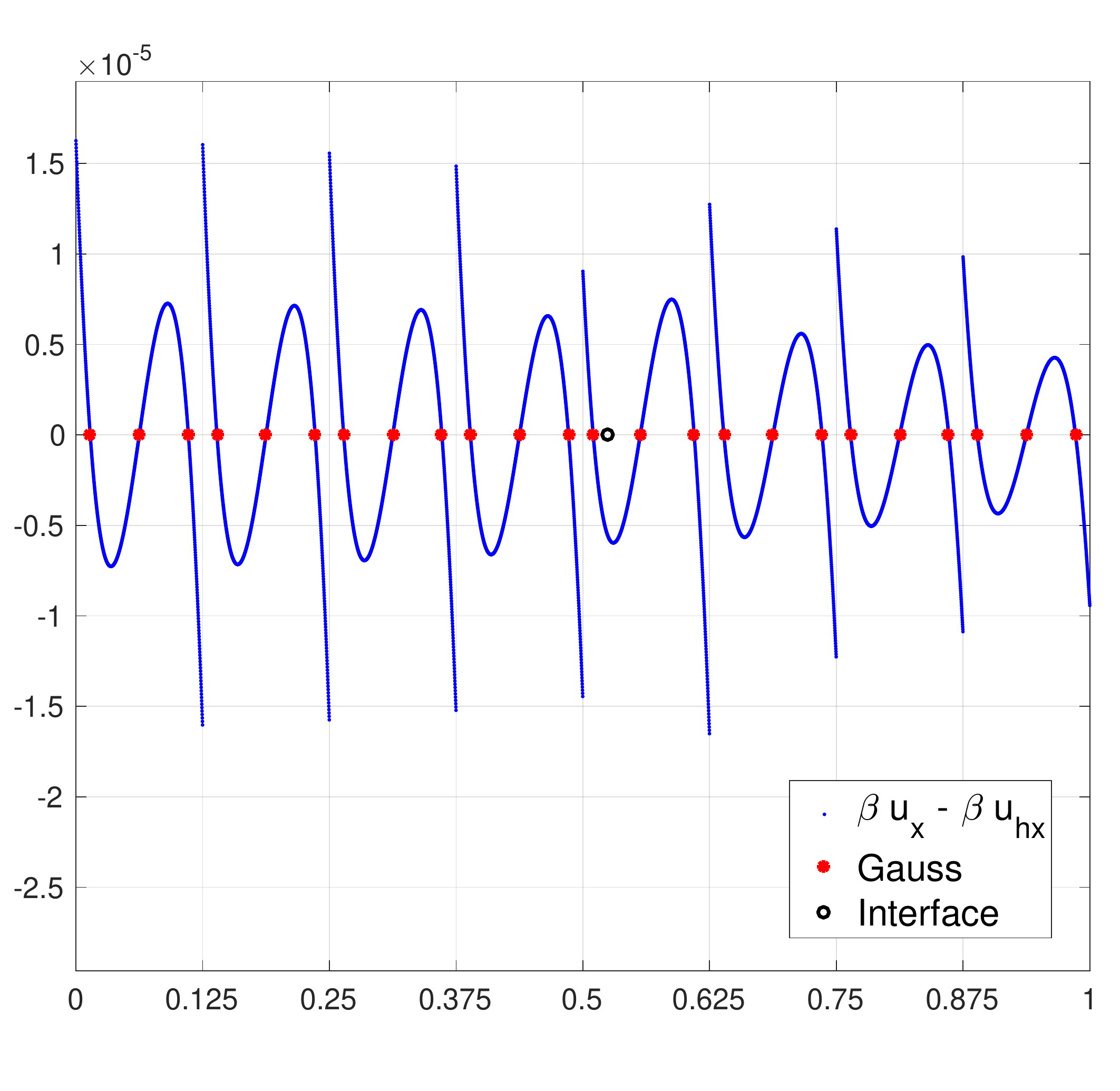}
  \caption{Error and flux error of $P_3$ IFVM solution. $\beta=\{1,5\}$, $\alpha = \dfrac{\pi}{6}$}
  \label{fig: error P3 IFV}
\end{figure}

\begin{example} (Superconvergence for less smooth functions).\end{example}
In the example, we test the convergence and superconvergence behivior for IFVM  and IFEM for nonsmooth functions.

For this example, we consider the following function as the exact solution
\begin{equation}\label{eq: IFE ex3}
  u(x) =
  \left\{
    \begin{array}{ll}
      \dfrac{1}{\beta^-}\cos(x), & \text{if}~~x \in [0,\alpha), \\
      \dfrac{1}{\beta^+}\cos(x) + \left(\dfrac{1}{\beta^-} - \dfrac{1}{\beta^+}\right)\cos(\alpha) + \dfrac{1}{\beta^+}(x-\alpha)^m, & \text{if}~~x \in (\alpha,1], \\
    \end{array}
  \right.
\end{equation}
where $m\ge 2$ is a positive integer. Direct calculation yields,
$$\bigjump{\beta u^{(j)}(\alpha)}= 0, ~~~ 1\le j\le m-1,~~~~~~\text{and}~~~~~~\bigjump{\beta u^{(m)}(\alpha)}\ne 0.$$
In particular, when $m=2$, the function \eqref{eq: IFE ex3} satisfies only the minimal regularity requirement \eqref{eq: jump condition}, but not the regularity condition in Theorem \ref{theo:11}. We test the diffusion interface problems using both immersed finite volume method and the immersed finite element methods \cite{2017CaoZhangZhang}. The errors of IFVM and IFEM solutions are presented in Table \ref{table: diffusion P2-IFV-1-5 nonsmooth}, \ref{table: diffusion P2-IFE-1-5 nonsmooth}, respectively. We note that the superconvergence behavior at (generalized) Lobatto points and (generalized) Gauss points are both affected by the low regularity of the exact solution. However we may still observe some superconvergence behavior at these points, even though neither of these convergence rates come close to the maximum rates of convergence in the analysis for smooth solution.

Moreover, we plot the errors of solution and flux for IFVM and IFEM in Figure \ref{fig: error P2 IFV nonsmooth}  and \ref{fig: error P2 IFE nonsmooth}, respectively. We can observe that IFVM flux error at (generalized) Gauss points are much closer to zero than the IFEM solution, even for nonsmooth functions. However, IFEM solution seems more accurate than IFVM solution on noninterface elements. In particular, the numerical solution at the mesh points are still exact, and the error at Lobatto points are much closer to zero than other interior points. For IFVM, the solution error at Lobatto points seems not superconvergent on either the interface element and noninterface elements.

\begin{table}[thb]
\begin{center}
\caption{Error of $P_2$ IFVM for Nonsmooth Solution $\beta=[1,5]$, $\alpha = \pi/6$, $\gamma=0$, $c=0$, $m=2$.}
\label{table: diffusion P2-IFV-1-5 nonsmooth}
\begin{small}
\begin{tabular}{|r|c|c|c|c|c|c|c|}
\hline
$1/h$ & $\|e_{\cal T}\|_{N}$ & $\|e_{\cal T}\|_{0,\infty}$ & $\|e_{\cal T}\|_{L}$& $\|\beta e_{\cal T}'\|_{G}$ &$\|e_{\cal T}\|_{0}$ & $|e_{\cal T}|_{1}$\\
\hline
    8 & 5.98e-05 & 1.61e-04 & 5.24e-05 & 1.19e-04 & 3.34e-05 & 2.24e-03  \\
  16 & 5.27e-05 & 1.19e-04 & 4.93e-05 & 1.05e-04 & 2.64e-05 & 1.40e-03  \\
  32 & 9.46e-06 & 9.96e-06 & 9.72e-06 & 1.89e-05 & 4.24e-06 & 1.56e-04  \\
  64 & 3.86e-06 & 6.71e-06 & 3.80e-06 & 7.49e-06 & 1.70e-06 & 1.73e-04  \\
128 & 2.20e-08 & 2.38e-08 & 2.18e-08 & 4.20e-08 & 9.35e-09 & 2.44e-06  \\
\hline
rate & 2.66        & 2.96        & 2.62        & 2.68        & 2.76         & 2.27  \\
\hline
\end{tabular}
\end{small}
\end{center}
\end{table}

\begin{table}[thb]
\begin{center}
\caption{Error of $P_2$ IFEM for Nonsmooth Solution $\beta=[1,5]$, $\alpha = \pi/6$, $\gamma=0$, $c=0$, $m=2$.}
\label{table: diffusion P2-IFE-1-5 nonsmooth}
\begin{small}
\begin{tabular}{|r|c|c|c|c|c|c|}
\hline
$1/h$ & $\|e_{\cal T}\|_{N}$ & $\|e_{\cal T}\|_{0,\infty}$ & $\|e_{\cal T}\|_{L}$& $\|\beta e_{\cal T}'\|_{G}$ &$\|e_{\cal T}\|_{0}$ & $|e_{\cal T}|_{1} $\\
\hline
    8 & 2.44e-15 & 1.49e-04 & 3.25e-05 & 4.21e-03 & 2.15e-05 & 1.93e-03  \\
  16 & 1.58e-14 & 7.42e-05 & 4.87e-06 & 4.40e-03 & 1.01e-05 & 1.28e-03  \\
  32 & 9.29e-14 & 8.34e-06 & 3.89e-06 & 1.54e-03 & 8.42e-07 & 1.88e-04  \\
  64 & 3.93e-13 & 4.45e-06 & 1.01e-06 & 5.02e-04 & 3.67e-07 & 1.61e-04  \\
128 & 8.00e-13 & 2.88e-08 & 4.23e-09 & 3.76e-05 & 9.02e-10 & 1.98e-06  \\
\hline
rate & -              & 3.00        & 2.62        & 2.68         & 3.39         & 2.29  \\
\hline
\end{tabular}
\end{small}
\end{center}
\end{table}

\begin{figure}[htb]
  \centering
  \includegraphics[width=.48\textwidth]{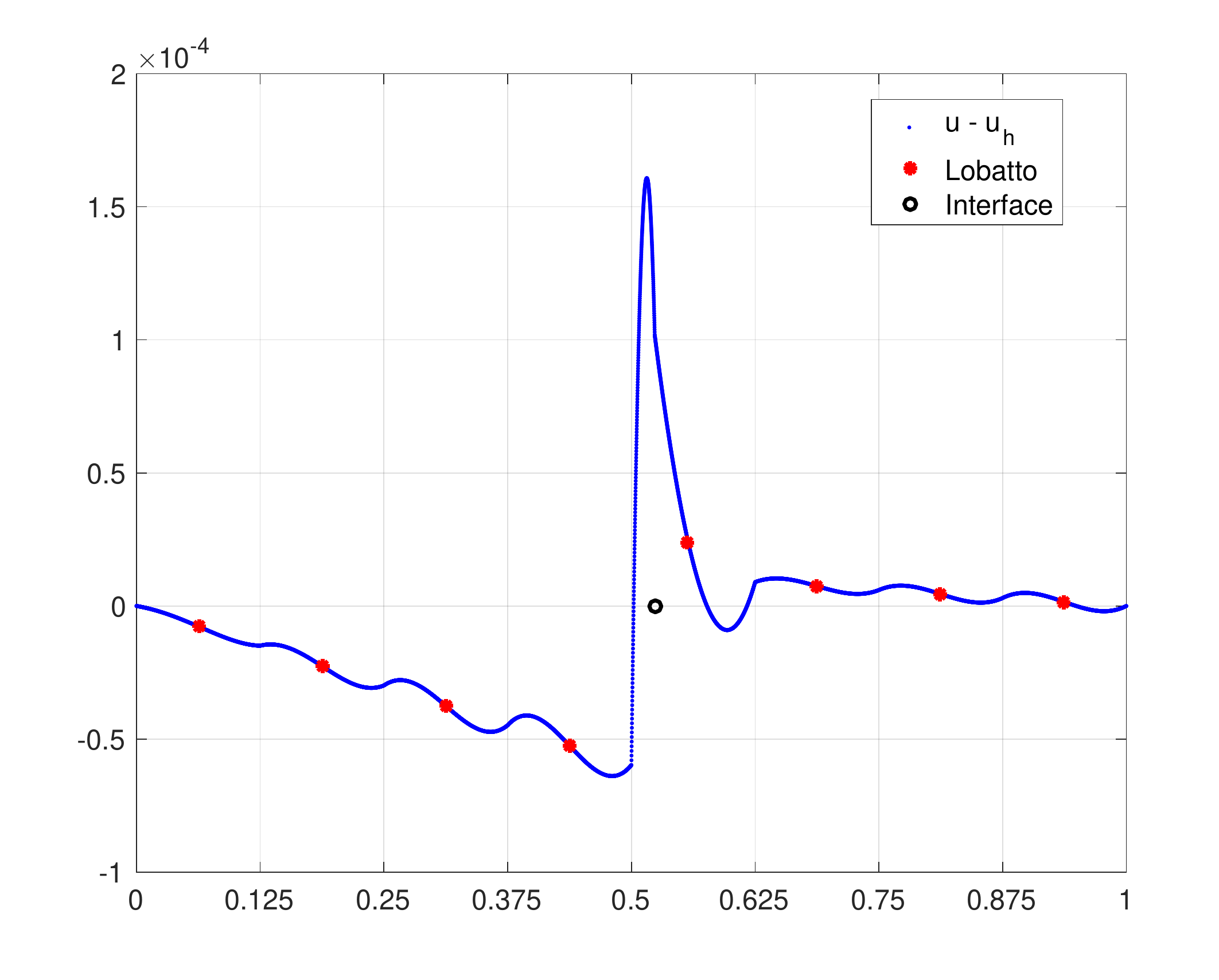}~
  \includegraphics[width=.48\textwidth]{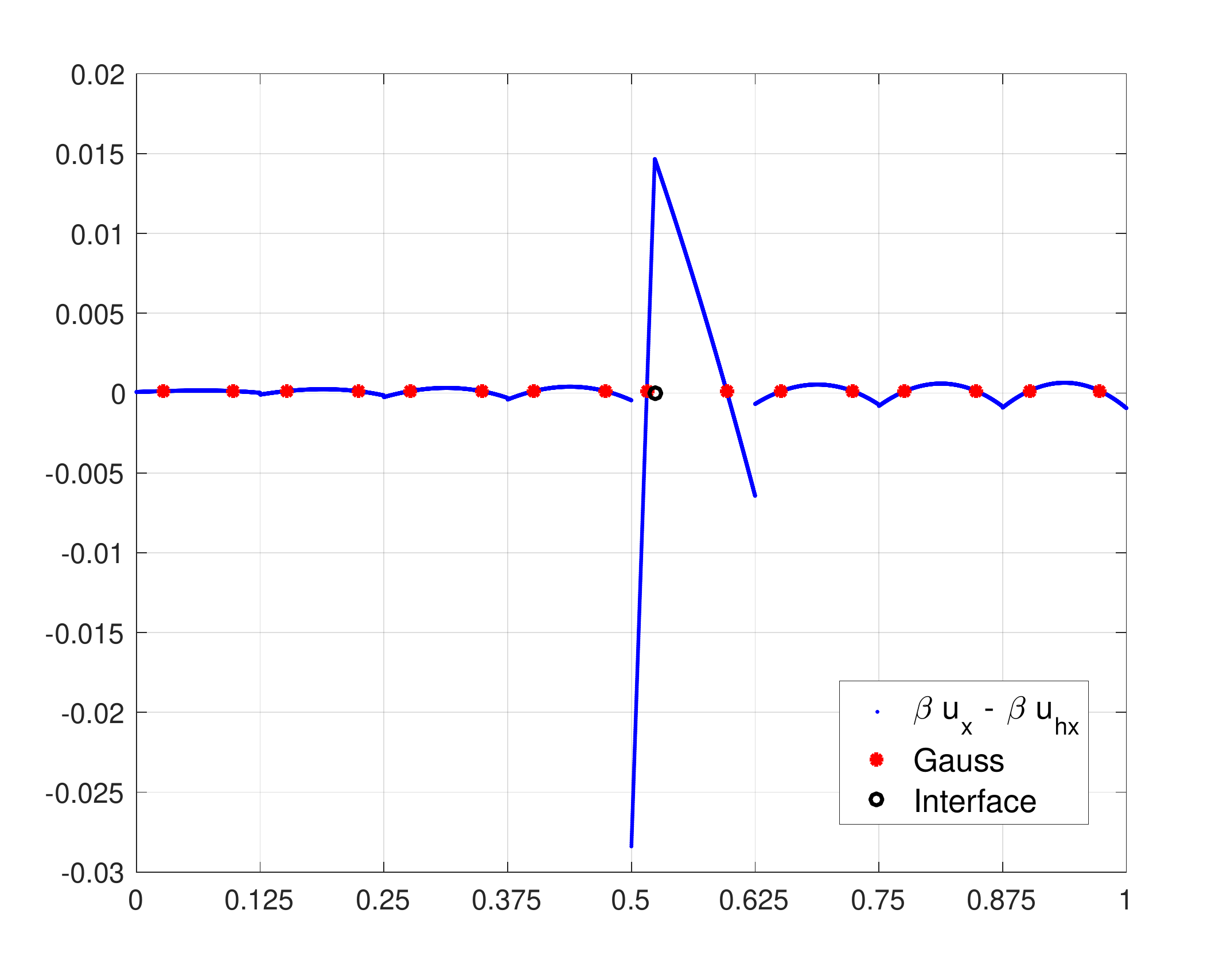}
  \caption{Error and flux error of $P_2$ IFVM solution for nonsmooth function. $\beta=\{1,5\}$, $\alpha = \dfrac{\pi}{6}$}
  \label{fig: error P2 IFV nonsmooth} 
\end{figure}

\begin{figure}[htb]
  \centering
  \includegraphics[width=.48\textwidth]{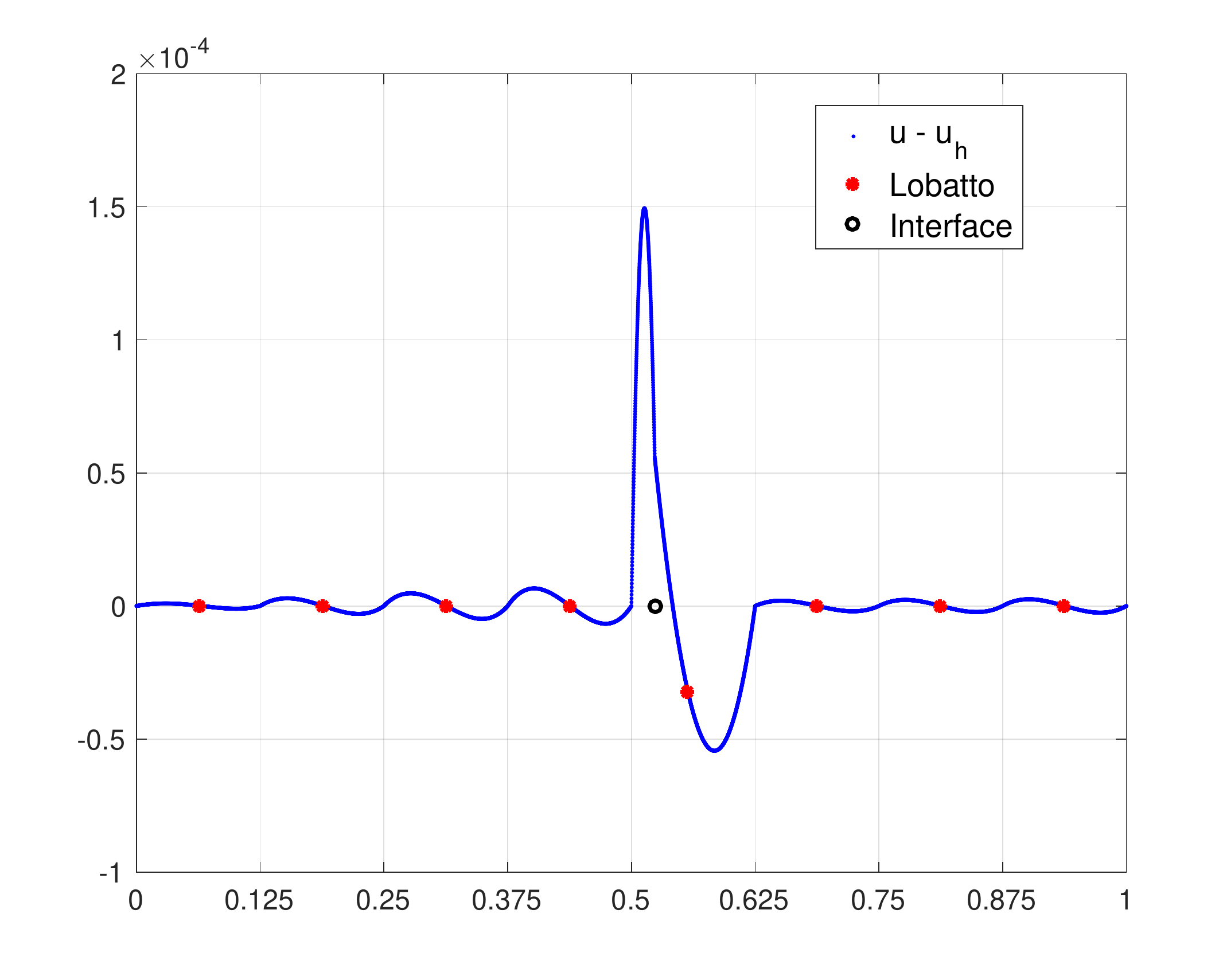}~
  \includegraphics[width=.48\textwidth]{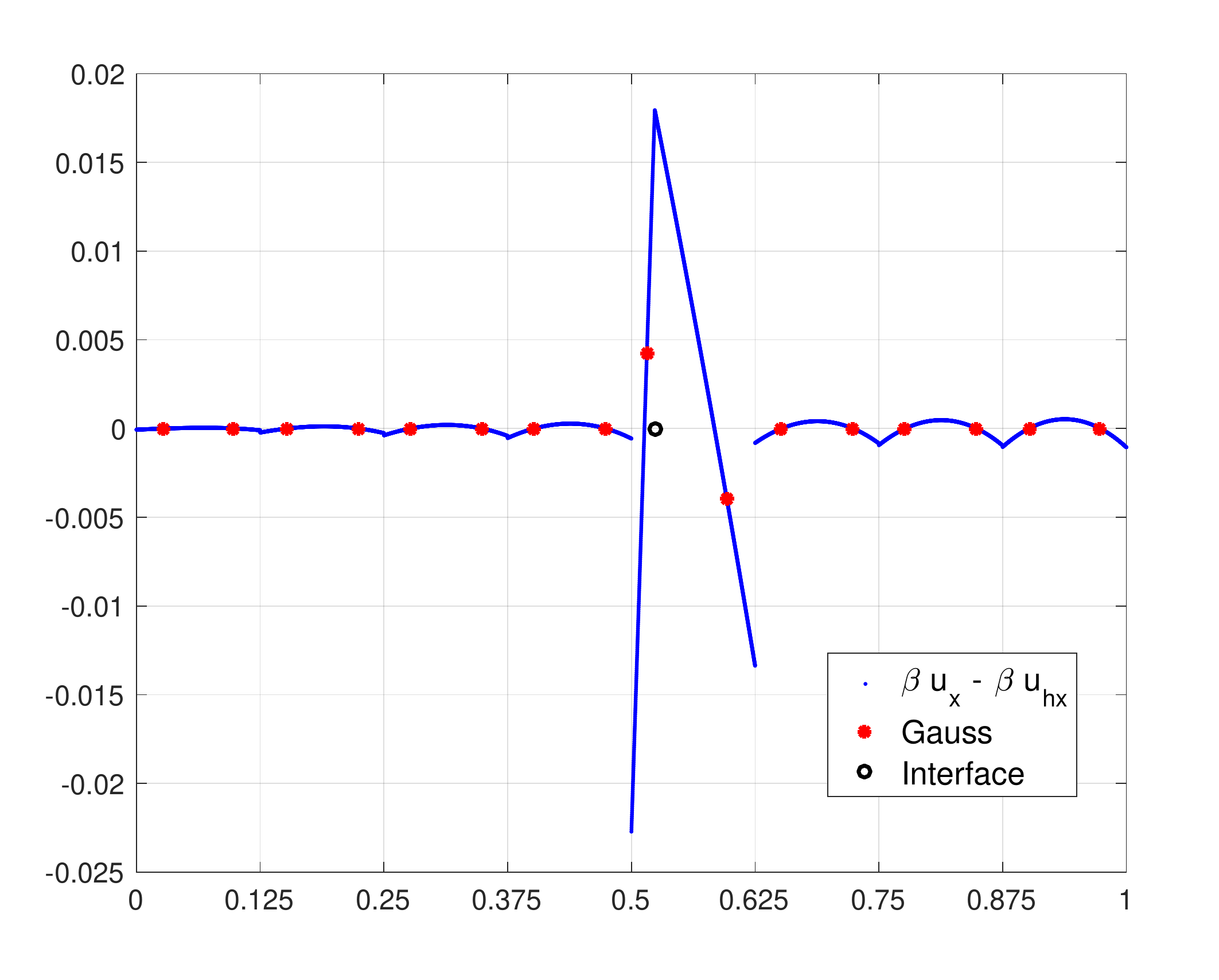}
  \caption{Error and flux error of $P_2$ IFEM solution for nonsmooth function. $\beta=\{1,5\}$, $\alpha = \dfrac{\pi}{6}$}
  \label{fig: error P2 IFE nonsmooth} 
\end{figure}

\section{Concluding Remarks}

   In this paper, we present an unified approach to study a class of high order IFVM for one-dimensional elliptic interface problems. 
 Using the generalized Lobatto polynomials which satisfy both orthogonality and interface jump conditions  as  the trial function space, and the generalized Gauss points  as 
 the control volume, we established the inf-sup condition and continuity of the bilinear form,  and then proved that the  IFVM solution converge optimally in both $H^1$- and $L^2$-norms.    Furthermore, we designed a new approach to study the superconvergence of IFVM, which is different from the method of Green function used in \cite{Cao;Zhang;Zou2012}, 
 and thus established superconvergence results for the IFV solution.  
 
 
The extension of the superconvergence analysis for two-dimensional interface problems is non-trivial. There are at least two obstacles. First, to the best of our knowledge, only the lowest order immersed finite element spaces ($P_1$ on triangular meshes and $Q_1$ on rectangular meshes) are reported for two-dimensional interface problems. The construction of higher order immersed FEM/FVM functions is still under investigation. Secondly, in two-dimensional case, the interface becomes an arbitrary curve, and in 3D, a surface. Error analysis for standard energy norm or $L^2$ norm is very difficult for such interface problems, and we believe the superconvergence analysis could even more challenging. Hence, the superconvergence analysis for multi-dimensional interface problems is a whole new territory, and therefore worth separate papers for dedicated study. 



\end{document}